\documentclass{article}
\usepackage[utf8]{inputenc}
\usepackage{authblk}
\usepackage{setspace}
\usepackage[margin=1in]{geometry}
\usepackage{graphicx}
\graphicspath{ {./figures/} }
\usepackage{subcaption}
\usepackage{amsmath}
\usepackage{lineno}
\usepackage{comment}
\usepackage{amsthm}

\newtheorem{thm}{Theorem}[section]

\newtheorem{lem}[thm]{Lemma}

\usepackage{amsfonts} 
\usepackage{mathtools} 

\usepackage{url}
\usepackage{hyperref}
\usepackage{indentfirst}
\usepackage{xcolor}
\usepackage{float}

\hypersetup{
    colorlinks,
    linkcolor={red!50!black},
    citecolor={blue!50!black},
    urlcolor={blue!80!black}
}
\usepackage[natbib=true,backend=biber,sorting=none,style=ieee, citestyle=numeric-comp]{biblatex}

\title{Role of Intra-specific Competition and Additional Food on Prey-Predator Systems exhibiting Holling Type-IV Functional Response}

\author[1]{D Bhanu Prakash}
\author[2]{D K K Vamsi}

\affil[1, 2]{ \ Department of Mathematics and Computer Science, Sri Sathya Sai Institute of Higher Learning, India.}
\affil[2]{Center for Excellence in Mathematical Biology, Sri Sathya Sai Institute of Higher Learning, India.}
\affil[1]{Corresponding Author. Email: dbhanuprakash@sssihl.edu.in}

\date{}

\begin{document}

\maketitle

\begin{abstract} {{
\noindent In recent years, the study on the impact of competition on additional food provided prey-predator systems have gained significant attention from researchers in the field of mathematical biology. In this study, we consider an additional food provided prey-predator model exhibiting Holling type-IV functional response and the intra-specific competition among predators. We prove the existence and uniqueness of global positive solutions for the proposed model. We study the existence and stability of equilibrium points and further explore the codimension-$1$ and $2$ bifurcations with respect to the additional food and competition. We further study the global dynamics of the system and discuss the consequences of providing additional food. Later, we do the time-optimal control studies with respect to the quality and quantity of additional food as control variables by transforming the independent variable in the control system. Making use of the Pontraygin maximum principle, we characterize the optimal quality of additional food and optimal quantity of additional food. We show that the findings of these dynamics and control studies have the potential to be applied to a variety of problems in pest management.
}}
\end{abstract}

{ \bf {keywords:} } Prey-Predator System; Holling type-IV response; Additional Food; Intra-specific Competition; Bifurcation; Time-optimal control; Pest management;



\section{Introduction} \label{Intro}

\indent Prey-predator models are mathematical representations used to study the intricate dynamics between two interacting species: the prey and its predator. These models aim to understand how changes in the population sizes of both species influence each other over time. By examining the fluctuations and stability of both populations, these models contribute to our understanding of ecological balance, the impact of external factors on species coexistence, and the potential consequences of perturbations in natural communities. Prey-predator models play a critical role in guiding conservation efforts, studying invasive species' effects, and comprehending the intricate web of life in ecosystems worldwide. 

One of the fundamental components in prey-predator models is the functional response which describes how the predator's consumption rate changes in response to variations in prey density \cite{kot2001elements}. The very first few proposed functional responses are the Holling functional responses, proposed by Canadian ecologist C.S. Holling in the 1950s \cite{metz2014dynamics}. The Holling type-IV functional response exhibits a saturation effect, meaning that the rate of prey consumption by predators increases at a decreasing rate as prey density increases. This saturation effect aligns with empirical observations that predators have limited capacity and cannot consume an unlimited number of prey items. 

In this study, we incorporate two of the major components that can alter the prey-predator dynamics: Additional Food and Competition. The concept of providing additional food to predators reflects a more realistic scenario where predators may have access to alternative food sources, such as other prey species or external resources. Understanding the role of additional food in prey-predator models is crucial for comprehending the complexity of ecological systems and the various factors that influence species coexistence and ecosystem stability. On the other hand, Competition in prey-predator systems arises when multiple predators target the same prey species or when prey species compete for limited resources. This competitive dynamic can significantly influence population stability and biodiversity. Intra-specific competition is a type of competition where the competition among predators is due to prey scarcity and this does not involve directly in the functional response unlike the case of mutual interference. Understanding these interactions is crucial in ecological modeling to predict population dynamics and inform conservation strategies.

Authors in \cite{V4DEDS,v2,V4Acta} studied the additional food provided prey-predator systems involving Holling type-III and Holling type-IV functional responses. The time-optimal control problems for additional food provided prey-predator systems involving Holling type-III and Holling type-IV functional responses are studied in \cite{ananth2021influence,ananth2022achieving,ananthcmb}. Recently, the stochastic time-optimal control problems for additional food provided prey-predator systems involving Holling type-III and Holling type-IV functional responses are studied in \cite{prakash2023stochastic,prakash9stochastic}. However, to the best of our knowledge, no work is available on prey-predator models incorporating both intra-specific competition and additional food to predators in the context of Holling type-IV functional response. In this work, we derive the prey-predator model and perform the dynamics and time-optimal control studies on the proposed system. 

The article is structured as follows: Section \ref{sec:4iscdmodel} introduces the prey-predator model with intra-specific competition and provision of additional food among predators. Section \ref{sec:4iscdposi} proves the positivity and boundedness of the solutions of the proposed system. Section \ref{sec:4iscdequil} investigates the conditions for the existence of various equilibria. The local stability of these equilibria is presented in section \ref{sec:4iscdstab}. In section \ref{sec:4iscdbifur}, we present the various possible local bifurcations exhibited by the proposed model both analytically and numerically. Section \ref{sec:4iscdglobaldynamics} studies the global dynamics of the proposed system in the parameter space of quality and quantity of additional food. Section \ref{sec:4iscdconseq} provides a detailed analysis of the consequences of providing additional food. Section \ref{sec:4iscdtimecontrol} presents the study on time-optimal control problems with quality or quantity of additional food as control parameters. Finally, we present the discussions and conclusions in section \ref{sec:disc}.

\section{Model Formulation} \label{sec:4iscdmodel}

Let $N$ and $P$ denote the biomass of prey and predator population densities respectively. In the absence of predator, the prey growth is modelled using logistic equation. Further, we assume that the prey species exhibit Holling type-IV functional response towards predators. Incorporating these assumptions, the prey-predator dynamics with Holling type-IV functional response can be described as:

\begin{equation} \label{4iscdnoa}
	\begin{split}
		\frac{\mathrm{d} N}{\mathrm{d} T} & = r N \left( 1-\frac{N}{K} \right)- \frac{c N P}{a(b N^2 + 1) + N}, \\
		\frac{\mathrm{d} P}{\mathrm{d} T} & = \frac{\delta_1 N P}{a(b N^2 + 1) + N} - m_1 P - d P^2.
	\end{split}
\end{equation}

This model is a generalized version of Bazykin model because in the absence of group defence of prey ($b$), this model is same as Bazykin model. We also assume that the predators are supplemented with an additional food of biomass A, which is uniformly distributed in the habitat. Accordingly, the system (\ref{4iscdnoa}) gets transformed to the following system.

\begin{equation} \label{4iscda}
	\begin{split}
		\frac{\mathrm{d} N}{\mathrm{d} T} & = r N \left( 1-\frac{N}{K} \right)- \frac{c N P}{(A \eta \alpha + a)(b N^2 + 1) + N}, \\
		\frac{\mathrm{d} P}{\mathrm{d} T} & = \delta_1 \left( \frac{N + \eta A (bN^2 + 1)}{(A \eta \alpha + a)(b N^2 + 1) + N} \right) P - m_1 P - d P^2.
	\end{split}
\end{equation}

Here the term $\eta$ represents the ratio between the search rate of the predator for additional food and prey respectively. The term $- d P^2(t)$ accounts for the intra-specific competition among the predators in order to avoid their unbounded growth in the absence of target prey \cite{V4Acta, V4DEDS}. Here the term $\alpha$ denotes the ratio between the maximum growth rates of the predator when it consumes the prey and additional food respectively. This term can be seen to be an equivalent of {\textit {\bf{quality}} } of additional food. For a complete analysis of functional responses of models (\ref{4iscdnoa}) and (\ref{4iscda}), the reader is advised to refer \cite{V4Acta}.

The biological descriptions of the various parameters involved in the systems  (\ref{4iscdnoa}) and  (\ref{4iscda}) are described in \autoref{4iscdparam_tab}.

\begin{table}[bht!]
	\centering
	\begin{tabular}{ccc}
		\hline
		Parameter & Definition & Dimension \\  
		\hline
		T & Time & time\\ 
		N & Prey density & biomass \\
		P & Predator density & biomass \\
		A & Additional food & biomass \\
		r & Prey intrinsic growth rate & time$^{-1}$ \\
		K & Prey carrying capacity & biomass \\
		c & Maximum rate of predation & time$^{-1}$ \\
		$\delta$ & Maximum growth rate of predator & time$^{-1}$ \\
		m & Predator mortality rate & time$^{-1}$ \\
		d & Death rate of predators & biomass$^{-1}$ time$^{-1}$ \\ 
		& due to intra-specific competition & \\
		$\alpha$ & Quality of additional food for predators & Dimensionless \\
		b & Group defence in prey & biomass$^{-2}$ \\
		\hline
	\end{tabular}
	\caption{Description of variables and parameters present in the systems (\ref{4iscdnoa}) and (\ref{4iscda}).}
	\label{4iscdparam_tab}
\end{table}

In order to reduce the complexity of the model, we non-dimensionalize the systems (\ref{4iscdnoa}) and (\ref{4iscda}) using the following non-dimensional parameters. 
$$N=ax, \  P=\frac{ary}{c}, \ t = r T,\ \gamma = \frac{K}{a}, \ \xi = \frac{\eta A}{a}, \ \omega = b a^2, \  \epsilon = \frac{c}{a d},\ \delta = \frac{\delta_1 a r}{c},\ m = \frac{m_1}{r}.$$
Accordingly, system (\ref{4iscdnoa}) gets reduced to the following system.
\begin{equation} \label{4iscdi}
	\begin{split}
		\frac{\mathrm{d} x}{\mathrm{d} t} & = x \left(1-\frac{x}{\gamma} \right)- \frac{xy}{\omega x^2 + 1 + x}, \\
		\frac{\mathrm{d} y}{\mathrm{d} t} & = \frac{\delta x y}{\omega x^2 + 1 + x} - m y - \epsilon y^2.
	\end{split}
\end{equation}

Also, the system (\ref{4iscda}) gets transformed to:
\begin{equation} \label{4iscd}
	\begin{split}
		\frac{\mathrm{d} x}{\mathrm{d} t} & = x \left(1-\frac{x}{\gamma} \right)- \frac{xy}{(1+\alpha \xi)(\omega x^2 + 1) + x}, \\
		\frac{\mathrm{d} y}{\mathrm{d} t} & = \frac{\delta \left(x + \xi (\omega x^2 + 1)\right) y}{(1 + \alpha \xi)(\omega x^2 + 1) + x} - m y - \epsilon y^2.
	\end{split}
\end{equation}
Here the term $\frac{\eta A^2}{N}$ denotes the quantity of additional food perceptible to the predator with respect to the prey relative to the nutritional value of prey to the additional food. Hence the term $\xi = \frac{\eta A}{a}$ can be seen to be an equivalent of {\textit {\bf{quantity}} }of additional food.

\section{Positivity and boundedness of the solution} \label{sec:4iscdposi}

\subsection{Positivity of the solution}

In this section, we demonstrate that the positive $xy$-quadrant is an invariant region for the system (\ref{4iscd}). Specifically, this means that if the initial populations of both prey and predator start in the positive $xy$-quadrant (i.e., $x(0) > 0$ and $y(0) > 0$), they will remain within this quadrant for all future times.

If prey population goes to zero (i.e., $x(t)=0$), then it is observed from the model equations (\ref{4iscd}) that $\frac{\mathrm{d} x}{\mathrm{d} t} = 0.$ This means that the prey population is constant (remains at zero) and cannot be negative. This holds even for the case when predator population goes to zero (i.e., $y=0$). Notably, $x=0$ and $y=0$ serve as invariant manifolds, with $\frac{\mathrm{d} x}{\mathrm{d} t} \Big|_{x=0} = 0$ and $\frac{\mathrm{d} y}{\mathrm{d} t} \Big|_{y=0} = 0$. Therefore, if a solution initiates within the confines of the positive $xy$-quadrant, it will either remains positive or stays at zero eternally (i.e., $x(t) \geq 0$ and $y(t) \geq 0 \ \forall t>0$ if $x(0)>0$ and $y(0)>0$). 

\subsection{Boundedness of the solution}

\begin{thm}
	Every solution of the system (\ref{4iscd}) that starts within the positive quadrant of the state space remains bounded. \label{4iscdbound}
\end{thm} 

\begin{proof}
	We define $W = x + \frac{1}{\delta}y$. Now, for any $K > 0$, we consider,
	\begin{equation*}
		\begin{split}
			\frac{\mathrm{d} W}{\mathrm{d} t} + K W = & \frac{\mathrm{d} x}{\mathrm{d} t} + \frac{1}{\delta} \frac{\mathrm{d} y}{\mathrm{d} t} + K x + \frac{K}{\delta} y \\
			& = x \left(1-\frac{x}{\gamma} \right)- \frac{xy}{(1+\alpha \xi)(\omega x^2 + 1) + x} \\
			&\  +\frac{1}{\delta}  \left( \frac{\delta \left(x + \xi (\omega x^2 + 1)\right) y}{(1 + \alpha \xi)(\omega x^2 + 1) + x} - m y - \epsilon y^2 \right) + K x + \frac{K}{\delta} y \\
			& =  x -\frac{x^2}{\gamma} - \frac{xy}{(1+\alpha \xi)(\omega x^2 + 1) + x} + \frac{\left(x + \xi (\omega x^2 + 1)\right) y}{(1 + \alpha \xi)(\omega x^2 + 1) + x} \\ 
			& \ - \frac{m}{\delta} y - \frac{\epsilon}{\delta} y^2 + K x + \frac{K}{\delta} y \\
			& =  x -\frac{x^2}{\gamma} + \frac{\xi (\omega x^2 + 1) y}{(1 + \alpha \xi)(\omega x^2 + 1) + x} - \frac{m}{\delta} y - \frac{\epsilon}{\delta} y^2 + K x + \frac{K}{\delta} y \\
			& \leq  x -\frac{x^2}{\gamma} + \frac{\xi y}{\alpha \xi+ 1} - \frac{m}{\delta} y - \frac{\epsilon}{\delta} y^2 + K x + \frac{K}{\delta} y \\
			& = (1+K) x -\frac{x^2}{\gamma} + \left( \frac{\xi}{1+\alpha \xi} + \frac{K}{\delta} - \frac{m}{\delta} \right) y - \frac{\epsilon}{\delta} y^2 \\
			& \leq \frac{\gamma (1+K)^2}{4} + \frac{\delta}{4 \epsilon} \left( \frac{\xi}{1+\alpha \xi} + \frac{K}{\delta} - \frac{m}{\delta} \right)^2 = M (say,)
		\end{split}
	\end{equation*}
	$$\frac{\mathrm{d} W}{\mathrm{d} t} + K W \leq M.$$
	Using Gronwall's inequality \cite{howard1998gronwall}, we now find an upper bound on $W(t)$. 
	
	This inequality is in the standard linear first-order form, and we solve it by multiplying both sides by an integrating factor $e^{Kt}$. This simplify the above inequality:
	$$\frac{\mathrm{d}}{\mathrm{d} t} (W(t) e^{Kt}) \leq M e^{Kt}.$$
	Now, integrating both sides from $0$ to $t$, we get
	$$0 \leq W(t) \leq \frac{M}{K} (1 - e^{-Kt}) + W(0) e^{-Kt}.$$
	Therefore, $0 < W(t) \leq \frac{M}{K}$ as $t \rightarrow \infty$. This demonstrates that the solutions of system (\ref{4iscd}) are ultimately bounded, thereby proving Theorem \autoref{4iscdbound}
\end{proof}

The \text{Picard-Lindel\"of theorem} guarentees the existence of a unique solution that exists locally in time for the system (\ref{4iscd}), given any initial conditions $x(0) = x_0 > 0$ and $y(0) = y_0 > 0$. This happens because the RHS terms in (\ref{4iscd}) are continuous and locally Lipschitz. Since the solution does not blow up in finite time (i.e., that the solution exists for all $t \geq 0$), global existence is also guaranteed.

\section{Existence of Equilibria} \label{sec:4iscdequil}

In this section, we investigate the existence of various equilibria that system (\ref{4iscd}) admits and study their stability nature. We first discuss the nature of nullclines of the considered system and the asymptotic behavior of its trajectories. We consider the biologically feasible parametric constraint $\delta > m$. 

The prey nullclines of the system (\ref{4iscd}) are given by 
$$ x = 0 , \  \  1-\frac{x}{\gamma} - \frac{y}{(1+\alpha \xi)(\omega x^2 + 1) + x} = 0. $$

The predator nullclines of the system (\ref{4iscd}) are given by
$$ y = 0, \ \ \frac{\delta \left(x + \xi (\omega x^2 + 1)\right)}{(1 + \alpha \xi)(\omega x^2 + 1) + x} - m - \epsilon y= 0. $$

Upon simplification, the non trivial prey nullcline is given as
\begin{equation} \label{4iscdprey}
	y = \left(1 - \frac{x}{\gamma} \right) \left( \left(1 + \alpha \xi \right) \left(\omega x^2 + 1 \right) + x\right).
\end{equation}

This nullcline is a smooth curve which is a cubic equation that passes through the point $(\gamma, 0)$ and $(0,1+\alpha \xi)$. Now, the slope of this nullcline is given by:
\begin{equation*}
	y'(x) = \left[1 - \frac{1 +\alpha \xi}{\gamma}\right]+2 \left[ \omega  \left( 1 + \alpha \xi \right) - \frac{1}{\gamma}\right] x - \frac{3 \omega}{\gamma} \left(1 + \alpha \xi \right) x^2.
\end{equation*}

This prey nullcline has negative slope at $x=\gamma$ (i.e.,$ y'(\gamma) = -1 - \gamma \omega (1+\alpha \xi) - \frac{1+\alpha \xi}{\gamma} < 0$). This nullcline has negative slope at $x=0$ if $\gamma < 1 + \alpha \xi$ and has a positive slope otherwise. The discriminant of the slope is given by 
\begin{equation*}
	\triangle = \frac{4}{\gamma^2} \left[ 1 + \gamma \omega \left(1 + \alpha \xi \right) + \omega \left(1+ \alpha \xi \right)^2 \left(\omega \gamma^2 - 3\right)\right].
\end{equation*} 

If $\omega \gamma^2 \geq 3$, then $\triangle > 0$. Therefore, prey nullcline has maximum and minimum in the interval $(0,\gamma)$ if $\omega \gamma^2 \geq 3$. Else this curve decreases monotonically from $0$ to $\gamma$.

We now have the following cases describing the existence of asymptotes for prey nullcline in the interval $[0,\gamma]$.

\begin{itemize}
	\item \textit{Case 1 $\left[ \gamma > 1 + \alpha \xi \right]$}: In this case, prey nullcline touches positive $y$-axis at $(0,1+\alpha \xi)$ and the slope increases from this point. After reaching a finite local maxima, the curve reaches $(0,\gamma)$.
	\item \textit{Case 2 $\left[ \gamma < 1 + \alpha \xi \text{ and } \omega \gamma^2 \geq 3 \right]$}: In this case, prey nullcline passes through $(0,1+\alpha \xi)$ and $(\gamma,0)$ and has a local maxima and local minima in the interval $[0,\gamma]$.
	\item \textit{Case 3 $\left[ \gamma < 1 + \alpha \xi \text{ and } \omega \gamma^2 < 3 \right]$}: In this case, prey nullcline monotonically decreases from $x=0$ to $x=\gamma$.
\end{itemize}

The non trivial predator nullcline is given as
$$ y = \frac{1}{\epsilon} \Bigg[ \frac{\delta \left(x + \xi (\omega x^2 + 1)\right)}{(1 + \alpha \xi)(\omega x^2 + 1) + x} - m \Bigg]. $$

In the absence of mutual interference (i.e., $\epsilon = 0$), these nullclines are straight lines parallel to $y$-axis. Upon simplification, the non trivial predator nullcline is
\begin{equation} \label{4iscdpredator}
	y = \frac{(\delta - m) x + \left(\delta \xi - m (1+\alpha \xi) \right) \left( \omega x^2 + 1 \right)}{ \epsilon \left( (1 + \alpha \xi)(\omega x^2 + 1) + x\right)}.
\end{equation}

This predator nullcline passes through $\left(0,\frac{\delta \xi - m (1+\alpha \xi)}{\epsilon (1+\alpha \xi)}\right)$. This point will be on the positive-$y$ axis if $\delta \xi - m (1+\alpha \xi) > 0$. 

Also this nullcline touches the $x$-axis at $$\left(\frac{-(\delta - m) \pm \sqrt{\triangle}}{2 \omega (\delta \xi - m(1+\alpha \xi))},0\right)\text{ where }\triangle  = (\delta - m)^2 - 4 \omega (\delta \xi - m(1+\alpha \xi))^2.$$ This point will be on the positive-$x$ axis if $\triangle > 0$ and $\delta \xi - m (1+\alpha \xi) < 0.$ These two conditions can be summarised as $\frac{-(\delta - m)}{2 \sqrt{\omega}} < \delta \xi - m (1+\alpha \xi) < 0$. It is also observed that the predator nullcline never reaches $\infty$ in finite time. 

The slope of the predator nullcline is given by
$$\frac{\mathrm{d} y}{\mathrm{d} x} =   \frac{\delta \left(1 - \omega x^2 \right) \left( 1 + \alpha \xi - \xi \right)}{ \epsilon \left(\left(1 + \alpha \xi \right) \left(\omega x^{2} + 1\right) + x\right)^{2}}.$$

The slope is positive when $x=0$ and  the slope is $0$ when $x = \pm \frac{1}{\sqrt{\omega}}$. From the second derivative test, predator nullcline attains maximum at $x=\frac{1}{\sqrt{\omega}}$ and the maximum value is given as $\left( \frac{\delta - m}{\sqrt{\omega}} \right) + 2 \left( \delta \xi - m (1+\alpha \xi) \right)  > 0$ i.e., $\delta \xi - m (1+\alpha \xi) > \frac{-(\delta - m)}{2 \sqrt{\omega}}$. Hence we have the sufficient condition for the existence of predator nullcline in the first quadrant is 
\begin{equation} \label{condition14iscd}
	\delta \xi - m (1+\alpha \xi) > \frac{-(\delta - m)}{2 \sqrt{\omega}}.
\end{equation}

The qualitative behavior of the predator nullcline can be understood in two scenarios:
\begin{itemize}
	\item \textit{Case P $\left[\delta \xi - m (1+\alpha \xi) > 0 \right]$}: Predator nullcline touches only positive $y$-axis and not pass through positive $x$-axis. 
	\item \textit{Case Q $\left[0 > \delta \xi - m (1+\alpha \xi) > \frac{-(\delta - m)}{2 \sqrt{\omega}} \right]$}: Predator nullcline touches only positive $x$-axis and not pass through positive $y$-axis. 
\end{itemize}

The possible configurations for the 1-3 cases of prey nullcline with cases P-Q of predator nullclines is presented in \autoref{nullclineplot4iscd}. The subfigures $A-C$ in \autoref{nullclineplot4iscd} represent the intersection of cases $1-3$ of prey nullcline (represented in solid blue line) with the case $P$ of the predator nullcline (represented in solid green line). The subfigures $D-F$ represent the similar scenarios for the case $Q$ of the predator nullcline. From this figure, it is observed that the interior equillibrium point exists for the system (\ref{4iscd}).

\begin{figure}[ht]
	\includegraphics[width=\textwidth]{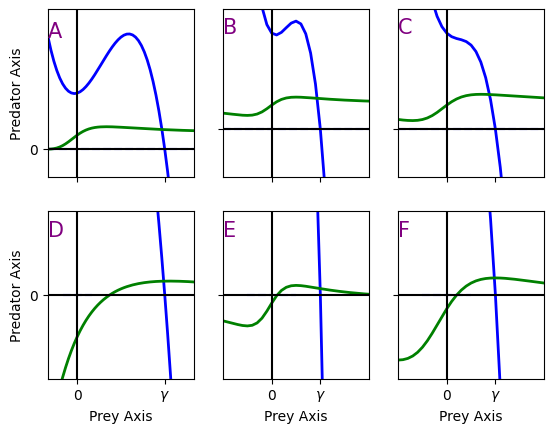}
	\caption{The possible configurations for the prey and predator nullclines of the system (\ref{4iscd}).}
	\label{nullclineplot4iscd}
\end{figure}

The intersection of these prey and predator nullclines will result in the following equilibria for the system (\ref{4iscd}).

\begin{itemize}
	\item Trivial equilibrium $E_0 = (0,0).$ 
	\item Predator free equilibrium $E_1 = (\gamma,0).$
	\item Pest free equilibrium $E_2 = \left(0,\frac{\delta \xi - m (1+\alpha \xi)}{\epsilon  (1+\alpha \xi)}\right)$. 
	\item Interior equilibrium $E^* = (x^*,y^*).$
\end{itemize}

The trivial ($E_0$) and axial equilibria ($E_1$) always exist for the system (\ref{4iscd}). Whereas the another axial equilibrium $E_2$ will be in the positive $xy$-quadrant if and only if $\delta \xi - m (1+\alpha \xi) > 0$. Since $\delta > m$, this equilibrium will be in the positive $xy$-quadrant for small $\alpha$ and large $\xi$. In the absence of additional food, this axial equilibrium $E_2 = \left(0,\frac{-m}{\epsilon}\right)$ exists on the negative $y$-axis.

The interior equilibrium of the system (\ref{4iscd}), if exists, is given by $E^* = (x^*, y^*)$  which is the point of intersection of the non trivial prey nullcline (\ref{4iscdprey}) and the non trivial predator nullcline (\ref{4iscdpredator}).

Upon simplification, $E^* = (x^*,y^*)$ is given by
\begin{equation} \label{4iscdystar}
	y^* = \frac{(\delta - m) x^* + \left(\delta \xi - m (1+\alpha \xi) \right) \left( \omega {x^{*}}^2 + 1 \right)}{ \epsilon \left( (1 + \alpha \xi)(\omega {x^{*}}^2  + 1) + x^*\right)},
\end{equation}
and $x^*$ should satisfy the following fifth order equation
\begin{equation} \label{4iscdxstar}
	\begin{split}
		\frac{\epsilon \omega^2}{\gamma} \left(1+\alpha \xi \right)^2 x^5 + \frac{\epsilon \omega (1+\alpha \xi)}{\gamma} \left(2 - \gamma \omega \left(1 + \alpha \xi \right)\right) x^4 & \\
		+ \frac{\epsilon}{\gamma} \left( 1 + 2 \omega \left(1 + \alpha \xi \right)^2 - 2 \gamma \omega \left(1 + \alpha \xi \right)\right) x^3 & \\
		+ \left( \omega \left( \delta \xi - m (1 + \alpha \xi)\right) + \frac{2 \epsilon }{\gamma} (1+\alpha \xi) - \epsilon - 2 \epsilon \omega (1 + \alpha \xi)^2 \right) x^2 & \\
		+ \left( \delta - m - 2 \epsilon (1+\alpha \xi)  + \frac{\epsilon}{\gamma} (1+\alpha \xi)^2 \right) x & \\
		+  \delta \xi - m (1 + \alpha \xi) - \epsilon (1+\alpha \xi)^2 & = 0.
	\end{split}
\end{equation}

This equation has at most five real roots. However, the Abel–Ruffini theorem states that there is no solution in radicals to general polynomial equations of degree five or higher with arbitrary coefficients. The study of nullclines numerically depicted the existence of atleast one interior equilibrium. 

The results obtained so far can be summarized as 

\begin{lem} \label{4iscdintcond}
	The system (\ref{4iscd}) exhibits at least one interior equilibrium $(E^* = (x^*,y^*))$ which is a solution of the equations (\ref{4iscdystar}) and (\ref{4iscdxstar}) and satisfying the condition $\delta \xi - m (1+\alpha \xi) > \frac{-(\delta - m)}{2 \sqrt{\omega}}$.
\end{lem}

In the absence of mutual interference (i.e., $\epsilon = 0$), $x^*$ is a solution of the quadratic equation $$\omega (\delta \xi - m (1 + \alpha \xi)) x^2 + (\delta - m) x + \delta \xi - m (1 + \alpha \xi)  = 0$$ and $$y^* = \left(1-\frac{x^*}{\gamma}\right) \left[(1+\alpha \xi) (\omega {x^{*}}^2 + 1) + x^*\right].$$ It also satisfies the condition $\delta \xi - m (1+\alpha \xi) > \frac{-(\delta - m)}{2 \sqrt{\omega}}$.

\section{Stability of Equilibria} \label{sec:4iscdstab}

In order to obtain the asymptotic behavior of the trajectories of the system (\ref{4iscd}), the associated Jacobian matrix is given by 

$$J = \begin{bmatrix}
	\frac{\partial}{\partial x} f(x,y)  & \frac{\partial}{\partial y} f(x,y)\\
	\frac{\partial}{\partial x} g(x,y) & \frac{\partial}{\partial y} g(x,y)
\end{bmatrix},$$
where
\begin{eqnarray*}
	f(x,y) &=& x \left(1-\frac{x}{\gamma} \right)- \frac{xy}{(1+\alpha \xi)(\omega x^2 + 1)+x}, \\
	g(x,y) &=& \frac{\delta \left( x + \xi (\omega x^2 + 1)\right)}{(1+\alpha \xi)(\omega x^2 + 1)+x} y - m y - \epsilon y^2,
\end{eqnarray*}
and 
\begin{eqnarray*}
	\frac{\partial}{\partial x} f(x,y) &=& 1 - \frac{2 x}{\gamma} - \frac{y \left(\alpha \xi + 1\right) \left(1 - \omega x^2 \right)}{\left((1+\alpha \xi)(\omega x^2 + 1)+x\right)^{2}}, \\
	\frac{\partial}{\partial y} f(x,y) &=& \frac{- x}{(1+\alpha \xi)(\omega x^2 + 1)+x}, \\
	\frac{\partial}{\partial x} g(x,y) &=&\frac{\delta y \left(1 - \omega x^{2}\right) \left(1 + \alpha \xi - \xi \right)}{\left((1+\alpha \xi)(\omega x^2 + 1)+x\right)^{2}}, \\
	\frac{\partial}{\partial y} g(x,y) &=& \frac{\delta \left(x + \xi \left(\omega x^{2} + 1\right)\right)}{(1+\alpha \xi)(\omega x^2 + 1)+x} - 2 \epsilon y - m.
\end{eqnarray*}

At the trivial equilibrium $E_0 = (0,0)$, we obtain the jacobian as 

\begin{equation*}
	J\left( E_0 \right) = \begin{bmatrix}
		1  & 0 \\
		0 & \frac{\delta \xi - m (1+ \alpha \xi)}{1+\alpha \xi}
	\end{bmatrix}.
\end{equation*}

The eigen values of this jacobian matrix are $1, \frac{\delta \xi - m (1+ \alpha \xi)}{1+\alpha \xi}$. If $\delta \xi - m (1+ \alpha \xi) > 0$, then both the eigen values have same signs. This makes the equilibrium $E_0 = (0,0)$ unstable. If $\delta \xi - m (1+ \alpha \xi) < 0$, then both the eigen values will have opposite signs. This makes the point $E_0 = (0,0)$ a saddle point. In the absence of additional food, $E_0 = (0,0)$ is a saddle point. 

At the axial equilibrium $E_1 = (\gamma ,0)$, we obtain the jacobian as 

\begin{equation*}
	J(E_1) = \begin{bmatrix}
		-1  & \frac{-\gamma}{(1+ \alpha \xi)(\omega \gamma^2 + 1) +\gamma} \\
		0 & \frac{(\delta - m) \gamma + (\delta \xi - m (1+\alpha \xi))(\omega \gamma^2 + 1)}{(1+ \alpha \xi)(\omega \gamma^2 + 1)+\gamma}
	\end{bmatrix}.
\end{equation*}

The eigen values of this jacobian matrix are $$-1,\  \frac{(\delta - m) \gamma + (\delta \xi - m (1+\alpha \xi))(\omega \gamma^2 + 1)}{(1+ \alpha \xi)(\omega \gamma^2 + 1)+\gamma}.$$

If $(\delta - m) \gamma + (\delta \xi - m (1+\alpha \xi))(\omega \gamma^2 + 1) > 0$, then both the eigenvalues will have opposite sign resulting in a saddle point. If $(\delta - m) \gamma + (\delta \xi - m (1+\alpha \xi))(\omega \gamma^2 + 1) < 0$, then it will be an asymptotically stable node. In the absence of additional food,  $E_1 = (\gamma, 0)$ is a saddle point if $\frac{\delta}{m} > 1 + \frac{\omega \gamma^2 + 1}{\gamma}$. Else it is a stable equilibrium. 

We now consider another axial equilibrium which exists only for the additional food provided system (\ref{4iscd}). At this axial equilibrium $E_2 = (0, \frac{\delta \xi - m (1+\alpha \xi)}{\epsilon (1+\alpha \xi)})$, the associated jacobian matrix is given as 

\begin{equation*}
	J(E_2) = \begin{bmatrix}
		1-\frac{\delta \xi - m (1+\alpha \xi)}{\epsilon (1+\alpha \xi)^2}  & 0 \\
		\frac{\delta}{\epsilon (1+\alpha \xi)^3} \left(\delta \xi - m (1+\alpha \xi ) \right) \left(1+(\alpha - 1) \xi \right)& \frac{-(\delta \xi - m (1+\alpha \xi))}{1+\alpha \xi}
	\end{bmatrix}.
\end{equation*}

The eigen values for this jacobian matrix are 
$$1-\frac{\delta \xi - m (1+\alpha \xi)}{\epsilon (1+\alpha \xi)^2},\ \frac{-(\delta \xi - m (1+\alpha \xi))}{1+\alpha \xi}.$$ 

Since this equilibrium exists in positive $xy$-quadrant only when $\delta \xi - m (1+\alpha \xi) > 0$, the second eigenvalue of the associated jacobian matrix is always negative. Therefore, the two eigenvalues will have same negative sign and result in a stable equilibrium when $\delta \xi - m (1+\alpha \xi) > \epsilon (1+\alpha \xi)^2 > 0$. If $\epsilon (1+\alpha \xi)^2 > \delta \xi - m (1+\alpha \xi) > 0$, then eigenvalues are of opposite sign resulting in a saddle equilibrium. 

The following lemmas present the stability nature of the trivial and axial equilibria. 

\begin{lem}
	The trivial equilibrium $E_0 = (0,0)$ is saddle (unstable node) if 
	\begin{equation*}
		\delta \xi - m (1+\alpha \xi) < (>)\  0.
	\end{equation*}
\end{lem}

\begin{lem}
	The predator-free axial equilibrium $E_1 = (\gamma,0)$ is stable node (saddle) if 
	\begin{equation*}
		\delta \xi - m (1+\alpha \xi) < (>)\  \frac{-(\delta - m) \gamma}{\omega \gamma^2 + 1}.
	\end{equation*}
\end{lem}

\begin{lem}
	The axial equilibrium $E_2 = \left(0, \frac{\delta \xi - m (1+\alpha \xi)}{\epsilon (1+\alpha \xi)}\right) $ exists in positive $xy$-quadrant and is stable node (saddle) if 
	\begin{equation*}
		\delta \xi - m (1+\alpha \xi) >\  0 \text{ and }\delta \xi - m (1+\alpha \xi) > (<) \epsilon (1+\alpha \xi)^2.
	\end{equation*}
\end{lem}

\subsection{Stability of Interior Equilibrium}

The interior equilibrium $E^* = (x^*,y^*)$ is the solution of system of equations  (\ref{4iscdystar}) and (\ref{4iscdxstar}) and satisfying the conditions in Lemma \ref{4iscdintcond}.

At this co existing equilibrium $E^* = (x^*, y^*)$, we obtain the jacobian as 

$$ J(E^*) = \begin{bmatrix}
	\frac{\partial}{\partial x} f(x,y)  & \frac{\partial}{\partial y} f(x,y) \\
	\frac{\partial}{\partial x} g(x,y) & \frac{\partial}{\partial y} g(x,y)
\end{bmatrix} \Bigg|_{(x^*,y^*)} . $$

The associated characteristic equation is given by

\begin{equation}
	\lambda ^2 - \text{Tr } J \bigg|_{(x^*,y^*)} \lambda + \text{Det } J \bigg|_{(x^*,y^*)} = 0.
\end{equation}

Now
\begin{eqnarray*}
	\text{Det } J \big|_{(x^*,y^*)} &=& \frac{x \left(\frac{\delta y \left(x + \xi \left(\omega x^{2} + 1\right)\right) \left(- 2 \omega x \left(\alpha \xi + 1\right) - 1\right)}{\left(\left(1 + \alpha \xi\right) \left(\omega x^{2} + 1\right) + x\right)^{2}} + \frac{\delta y \left(2 \omega x \xi + 1\right)}{\left(1 + \alpha \xi\right) \left(\omega x^{2} + 1\right) + x}\right)}{\left(1 + \alpha \xi\right) \left(\omega x^{2} + 1\right) + x} \\
	& & + \left(\frac{\delta \left(x + \xi \left(\omega x^{2} + 1\right)\right)}{\left(1 + \alpha \xi\right) \left(\omega x^{2} + 1\right) + x} - 2 \epsilon y - m\right) \\
	& & \bigg(- \frac{x y \left(- 2 \omega x \left(\alpha \xi + 1\right) - 1\right)}{\left(\left(1 + \alpha \xi\right) \left(\omega x^{2} + 1\right) + x\right)^{2}} \\
	& & - \frac{y}{\left(1 + \alpha \xi\right) \left(\omega x^{2} + 1\right) + x} + 1 - \frac{2 x}{\gamma}\bigg) \bigg|_{(x^*,y^*)}.
\end{eqnarray*}

From (\ref{4iscd}), the following equations satisfy at the interior equilibrium $E^* = (x^*,y^*)$.
$$ 1-\frac{x}{\gamma} = \frac{y}{\left(1 + \alpha \xi\right) \left(\omega x^{2} + 1\right) + x}\  \bigg|_{(x^*,y^*)}.$$ and $$ \frac{\delta \left(x + \xi (\omega x^2 + 1) \right)}{\left(1 + \alpha \xi\right) \left(\omega x^{2} + 1\right) + x} = m + \epsilon y \  \bigg|_{(x^*,y^*)}.$$

Substituting these two equations in the definition of determinent, we have

\begin{eqnarray*}
	\text{Det } J \big|_{(x^*,y^*)} &=& \frac{xy \left(\frac{\left( m + \epsilon y \right) \left(- 2 \omega x \left(\alpha \xi + 1\right) - 1\right)}{\left(1 + \alpha \xi\right) \left(\omega x^{2} + 1\right) + x} + \frac{\delta \left(2 \omega x \xi + 1\right)}{\left(1 + \alpha \xi\right) \left(\omega x^{2} + 1\right) + x}\right)}{\left(1 + \alpha \xi\right) \left(\omega x^{2} + 1\right) + x} + \left( m + \epsilon y - 2 \epsilon y - m\right) \\
	& & \bigg(- \frac{x y \left(- 2 \omega x \left(\alpha \xi + 1\right) - 1\right)}{\left(\left(1 + \alpha \xi\right) \left(\omega x^{2} + 1\right) + x\right)^{2}} - 1 + \frac{x}{\gamma} + 1 - \frac{2 x}{\gamma}\bigg) \bigg|_{(x^*,y^*)}.
\end{eqnarray*}

Upon simplification, we have
\begin{eqnarray*}
	& & \text{Det } J \big|_{(x^*,y^*)} \\
	&=& \frac{x y \left(\left( m + 2 \epsilon y \right) \left(- 2 \omega x \left(\alpha \xi + 1\right) - 1\right)+ \delta \left(2 \omega x \xi + 1\right)\right)}{\left(\left(1 + \alpha \xi\right) \left(\omega x^{2} + 1\right) + x \right)^2} + \frac{\epsilon xy}{\gamma} \bigg|_{*} \\
	&=& \frac{x y \left( \delta - m + 2 \omega x \left(\delta \xi - m (1 + \alpha \xi )\right) - 2 \epsilon y \left( 2 \omega x (1+\alpha \xi ) + 1 \right)\right)}{\left(\left(1 + \alpha \xi\right) \left(\omega x^{2} + 1\right) + x \right)^2} + \frac{\epsilon xy}{\gamma} \bigg|_{*}.
\end{eqnarray*}

In the absence of mutual interference (i.e., $\epsilon = 0$), the determinent is positive when $\delta \xi - m(1+\alpha \xi) > \frac{-(\delta - m)}{2 \omega x^*}$. In the presence of mutual interference, the determinent is positive only when 

\begin{equation} \label{4iscdintdet}
	0 < \epsilon < \frac{\delta - m + 2 \omega x^* (\delta \xi - m (1+\alpha \xi))}{2 y^* (2 \omega x^*  (1+\alpha \xi)+1)}.
\end{equation}

The trace of the jacobian matrix is given by 

\begin{eqnarray*}
	\text{Tr } J \bigg|_{(x^*,y^*)}  &=& \frac{\delta \left(x + \xi \left(\omega x^{2} + 1\right)\right)}{\left(1 + \alpha \xi\right) \left(\omega x^{2} + 1\right) + x} - 2 \epsilon y - m + 1 - \frac{2 x}{\gamma} \\
	& & - \frac{x y \left(- 2 \omega x \left(\alpha \xi + 1\right) - 1\right)}{\left(\left(1 + \alpha \xi\right) \left(\omega x^{2} + 1\right) + x\right)^{2}} - \frac{y}{\left(1 + \alpha \xi\right) \left(\omega x^{2} + 1\right) + x} \bigg|_{(x^*,y^*)}.
\end{eqnarray*}

From (\ref{4iscd}), the following equations satisfy at the interior equilibrium $E^* = (x^*,y^*)$.
$$ 1-\frac{x}{\gamma} = \frac{y}{\left(1 + \alpha \xi\right) \left(\omega x^{2} + 1\right) + x} \bigg|_{(x^*,y^*)}.$$ and $$ \frac{\delta \left(x + \xi (\omega x^2 + 1) \right)}{\left(1 + \alpha \xi\right) \left(\omega x^{2} + 1\right) + x} = m + \epsilon y \bigg|_{(x^*,y^*)}.$$

Substituting these two equations in the trace of jacobian, we have

\begin{eqnarray*}
	\text{Tr } J \bigg|_{(x^*,y^*)}  &=& m + \epsilon y - 2 \epsilon y - m + 1 - \frac{2 x}{\gamma} \\
	& & + \left(1 - \frac{x}{\gamma} \right) \left( \frac{x \left(2 \omega x \left(\alpha \xi + 1\right) + 1\right)}{\left(1 + \alpha \xi\right) \left(\omega x^{2} + 1\right) + x} \right)  - 1 + \frac{x}{\gamma} \bigg|_{(x^*,y^*)}. \\
	&=& - \epsilon y - \frac{x}{\gamma} + \left(1 - \frac{x}{\gamma} \right) \left(\frac{x + 2 \omega x^2 \left(\alpha \xi + 1\right)}{\left(1 + \alpha \xi\right) \left(\omega x^{2} + 1\right) + x} \right)  \bigg|_{(x^*,y^*)}. \\
\end{eqnarray*}

Therefore, trace of the jacobian is negative when
\begin{eqnarray*}
	& & - \epsilon y - \frac{x}{\gamma} + \left(1 - \frac{x}{\gamma} \right) \frac{x + 2 \omega x^2 \left(\alpha \xi + 1\right)}{\left(1 + \alpha \xi\right) \left(\omega x^{2} + 1\right) + x}  \bigg|_{(x^*,y^*)}  < 0 \\
	\implies & & \epsilon y > - \frac{x}{\gamma} + \left(1 - \frac{x}{\gamma} \right) \frac{x + 2 \omega x^2 \left(\alpha \xi + 1\right)}{\left(1 + \alpha \xi\right) \left(\omega x^{2} + 1\right) + x} \bigg|_{(x^*,y^*)}  \\
	\implies & & \epsilon y > - \frac{x}{\gamma} + \left(1 - \frac{x}{\gamma} \right) \frac{2 \omega x^2}{\omega x^2 + 1} \bigg|_{(x^*,y^*)}.
\end{eqnarray*}

The trace of the jacobian is negative when 

\begin{equation} \label{4iscdinttrace}
	\epsilon > \frac{2 \gamma \omega {x^{*}}^2 - 3 \omega {x^{*}}^3 - x^*}{\gamma y^* \left(\omega {x^{*}}^2 + 1\right)}.
\end{equation}

The results obtained in this section can be summarized as follows:

\begin{thm}
	The interior equilibrium $(E^* = (x^*,y^*))$ of the system (\ref{4iscd}) exists when it satisfies the conditions in Lemma \ref{4iscdintcond}. The nature of this equilibrium depends on the signs of determinent and trace of the jacobian matrix given in equations (\ref{4iscdintdet}) and (\ref{4iscdinttrace}). The interior equilibrium is 
	\begin{itemize}
		\item asymptotically stable when $	0 < \epsilon < \frac{\delta - m + 2 \omega x^* (\delta \xi - m (1+\alpha \xi))}{2 y^* (2 \omega x^*  (1+\alpha \xi)+1)}$ and $	\epsilon > \frac{2 \gamma \omega {x^{*}}^2 - 3 \omega {x^{*}}^3 - x^*}{\gamma y^* \left(\omega {x^{*}}^2 + 1\right)}$,
		\item an unstable point when $	0 < \epsilon < \frac{\delta - m + 2 \omega x^* (\delta \xi - m (1+\alpha \xi))}{2 y^* (2 \omega x^*  (1+\alpha \xi)+1)}$ and $0 < \epsilon < \frac{2 \gamma \omega {x^{*}}^2 - 3 \omega {x^{*}}^3 - x^*}{\gamma y^* \left(\omega {x^{*}}^2 + 1\right)}$,
		\item a saddle point when  $\epsilon > \frac{\delta - m + 2 \omega x^* (\delta \xi - m (1+\alpha \xi))}{2 y^* (2 \omega x^*  (1+\alpha \xi)+1)}$.
	\end{itemize}
\end{thm}

\section{Bifurcation Analysis} \label{sec:4iscdbifur}

\subsection{Transcritical Bifurcation}
Within this subsection, we derive the conditions for the existence of transcritical bifurcation near the equilibrium point $E_1 = (\gamma,0)$ using the parameter $\xi$ as the bifurcation parameter.

\begin{thm}
	When the parameter satisfies $(1-\alpha) \gamma + \omega \gamma^2 + 1 \neq 0,\ \delta - m \alpha \neq 0$ and $\xi = \xi^* = \frac{m (\omega \gamma^2 + \gamma + 1)-\delta \gamma}{(\delta - m \alpha)(\omega \gamma^2 + 1)}$, a transcritical bifurcation occurs at $E_1 = (\gamma,0)$ in the system (\ref{4iscd}).
\end{thm}

\begin{proof}
	The jacobian matrix corresponding to equilibrium point $E_1=(\gamma,0)$ is given by
	
	\begin{equation*}
		J(E_1) = 
		\begin{bmatrix}
			-1  & \frac{-\gamma}{\gamma + (1 + \alpha \xi) (\omega \gamma^2 + 1)} \\\
			0 & \frac{(\delta - m) \gamma + (\delta \xi - m (1+\alpha \xi))(\omega \gamma^2 + 1)}{\gamma + (1 + \alpha \xi) (\omega \gamma^2 + 1)}
		\end{bmatrix}.
	\end{equation*}
	The eigenvectors corresponding to the zero eigenvalues of $J(E_1)$ and $J(E_1)^{T}$ be denoted by $V$ and $W$, respectively. 
	
	\begin{equation*}
		V = \begin{bmatrix} V_1  \\ V_2 \end{bmatrix} = \begin{bmatrix}	1  \\ - \frac{\gamma + (1 + \alpha \xi) (\omega \gamma^2 + 1)}{\gamma} \end{bmatrix} , \ \ W = \begin{bmatrix} 0 \\ 1 \end{bmatrix}.
	\end{equation*}
	Note that $V_2 < 0$. Let us denote system (\ref{4iscd}) as $H = \begin{bmatrix} F  \\ G \end{bmatrix}.$
	Thus, 
	
	\begin{eqnarray*}
		H_{\xi} (E_1; \xi^*) &=&  \begin{bmatrix} 0 \\ 0 \end{bmatrix}, \\
		DH_\xi(E_1; \xi^*)V &=& 
		\begin{bmatrix}
			\frac{\partial F_\xi}{\partial x} & \frac{\partial F_\xi}{\partial y} \\
			\frac{\partial G_\xi}{\partial x} & \frac{\partial G_\xi}{\partial y}
		\end{bmatrix}
		\begin{bmatrix}
			V_1 \\
			V_2
		\end{bmatrix}
		_{(E_1; \xi^*)} \\
		&=& \frac{-\alpha (\omega \gamma^2 + 1)}{\gamma + (\omega \gamma^2 + 1)(1+\alpha \xi)}
		\begin{bmatrix}
			1 \\
			\frac{\delta}{\alpha \gamma} \left((1-\alpha)\gamma + \omega \gamma^2 + 1\right)
		\end{bmatrix}, \\ 
		D^2 H(E_1; \xi^*)(V, V) &=&
		\begin{bmatrix}
			\frac{\partial^2 F}{\partial x^2} V_1^2 + 2 \frac{\partial^2 F}{\partial x \partial y} V_1 V_2 + \frac{\partial^2 F}{\partial y^2} V_2^2 \\
			\frac{\partial^2 G}{\partial x^2} V_1^2 + 2 \frac{\partial^2 G}{\partial x \partial y} V_1 V_2 + \frac{\partial^2 G}{\partial y^2} V_2^2
		\end{bmatrix}
		_{(E_1, \xi^*)} \\
		&=& \begin{bmatrix}
			-\frac{2}{\gamma} - \frac{2 (1 - \omega \gamma^2) (1 + \alpha \xi)}{(\gamma + (1 + \alpha \xi)(\omega \gamma^2 + 1))^2} \\
			- \frac{2 \delta (1 + \alpha \xi - \xi) ( 1 - \omega \gamma^2)}{\gamma (\gamma + (1 + \alpha \xi) (\omega \gamma^2 + 1))}- \frac{2 \epsilon}{\gamma^2} \left( \gamma + (1 + \alpha \xi) (\omega \gamma^2 + 1)\right)^2
		\end{bmatrix}.
	\end{eqnarray*}
	
	If $1+\omega \gamma^2 + (1-\alpha)\gamma \neq 0$, we have
	
	\begin{eqnarray*}
		W^{T} H_{\xi} (E_1; \xi^*)&=& 0, \\
		W^{T} [DH_\xi(E_1; \xi^*)V]&=& \frac{-\delta (\omega \gamma^2 + 1) (1+\omega \gamma^2 +(1-\alpha)\gamma)}{\gamma (\gamma + (1 + \alpha \xi) (1+\omega \gamma^2))} \neq 0, \\
		W^{T} [D^2 H(E_1; \xi^*)(V, V) ]&=& - \frac{2 \delta (1 + \alpha \xi - \xi) ( 1 - \omega \gamma^2)}{\gamma (\gamma + (1 + \alpha \xi) (\omega \gamma^2 + 1))}- \frac{2 \epsilon}{\gamma^2} \left( \gamma + (1 + \alpha \xi) (\omega \gamma^2 + 1)\right)^2 \\ & & \neq 0.
	\end{eqnarray*}
	By the Sotomayor's theorem \cite{perko2013differential}, the system (\ref{4iscd}) undergoes a transcritical bifurcation around $E_1$ at $\xi = \xi^*$.
\end{proof}

In \autoref{trans14iscd}, the nature of interior equilibrium shifts from saddle to stable and the the trivial equilibrium $E_0$ changes its nature around the same point from saddle to unstable node. This results in the transcritical bifurcation at the value $\xi =2.9$. \autoref{trans24iscd} represents the transcritical bifurcation where the stability of interior equilibrium ($E^* = (x^*,y^*)$) and the axial eqilibrium ($E_1 = (\gamma,0)$) are exchanged at the bifurcation point $\xi = 2.8$.  We depict the equilibria and their stability for the following set of parameter values, $\gamma = 1.0,\ \alpha = 1.0,\ \epsilon = 0.5,\ \delta = 8.0,\ m=6.0, \ \omega = 4.0$. The two subplots in \autoref{trans14iscd} - \autoref{saddle4iscd} represents the bifurcation diagrams for the parameter $\xi$ with respect to the prey and predator populations respectively. 

\begin{figure}[ht]
	\centering
	\includegraphics[width=\textwidth]{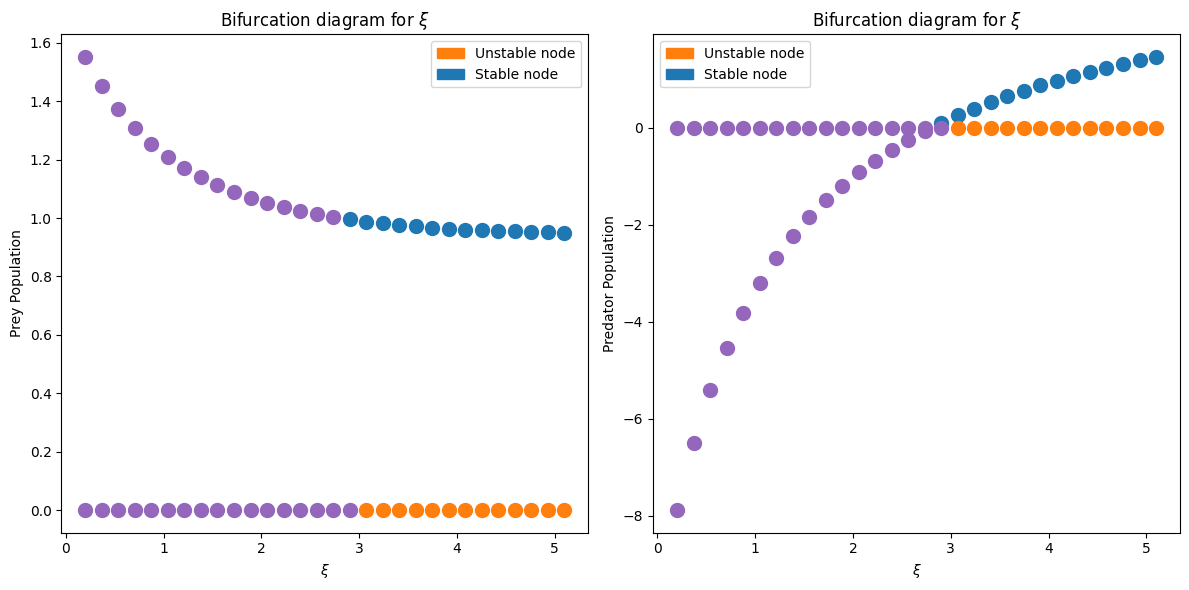}
	\caption{Transcritical bifurcation diagram around trivial equilibrium $E_0 = (0,0)$ with respect to the quantity of additional food $\xi$.}
	\label{trans14iscd}
\end{figure}

\begin{figure}[ht]
	\centering
	\includegraphics[width=\textwidth]{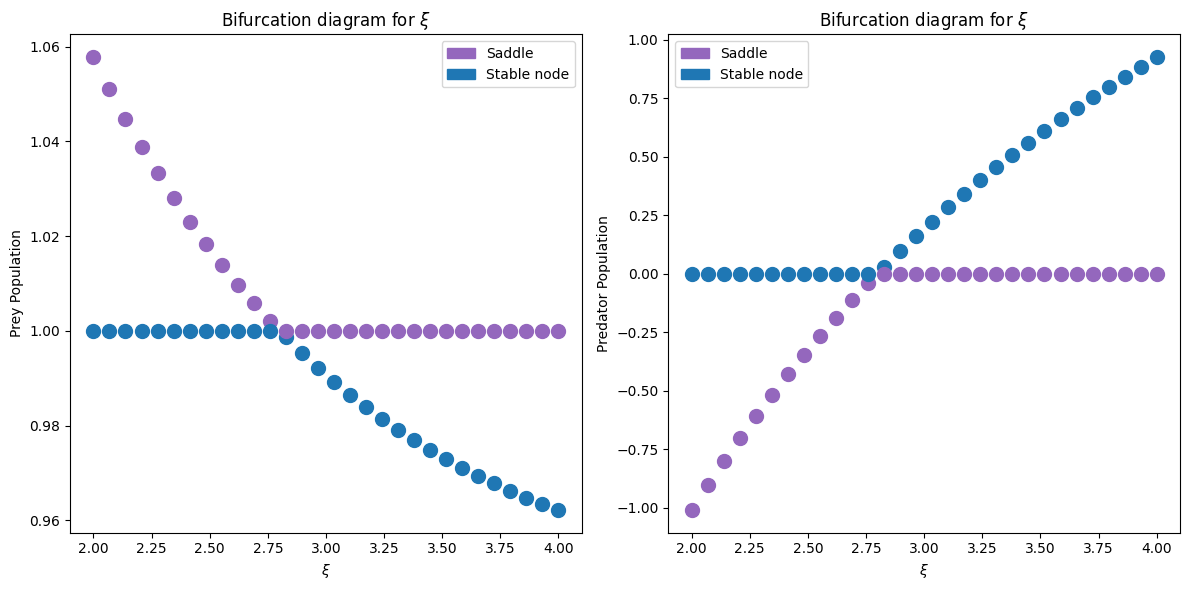}
	\caption{Transcritical bifurcation diagram around axial equilibrium $E_1 = (\gamma,0)$ with respect to the quantity of additional food $\xi$.}
	\label{trans24iscd}
\end{figure}

\subsection{Saddle-node Bifurcation}
Within this subsection, we derive the conditions for the existence of saddle-node bifurcation near the equilibrium point $E_2 = \left( 0, \frac{\delta \xi - m (1+\alpha \xi)}{\epsilon (1 + \alpha \xi)} \right)$ using the parameter $\xi$ as the bifurcation parameter.

\begin{thm}
	When the parameter satisfies $\delta \neq m \alpha,\ \epsilon (1+\alpha \xi)^2 \neq \delta \xi - m (1 + \alpha \xi)$ and $\xi = \xi^* = \frac{m}{\delta - m \alpha}$, a saddle-node bifurcation occurs at $E_2 = \left( 0, \frac{\delta \xi - m (1+\alpha \xi)}{\epsilon (1 + \alpha \xi)} \right)$ in the system (\ref{4iscd}).
\end{thm}

\begin{proof}
	The jacobian matrix corresponding to equilibrium point $E_2 = \left( 0, \frac{\delta \xi - m (1+\alpha \xi)}{\epsilon (1+\alpha \xi)} \right)$ is given by
	
	\begin{equation*}
		J(E_2) =
		\begin{bmatrix}
			1 - \frac{\delta \xi - m (1 + \alpha \xi)}{\epsilon (1 + \alpha \xi)^2} & 0 \\
			\frac{\delta (\delta \xi - m (1+\alpha \xi)) (1 + \alpha \xi - \xi)}{\epsilon (1 + \alpha \xi)^3} & - \frac{\delta \xi - m (1 + \alpha \xi)}{1 + \alpha \xi}
		\end{bmatrix}.
	\end{equation*}
	
	The eigenvectors corresponding to the zero eigenvalues of $J(E_2)$ and $J(E_2)^{T}$ be denoted by $V$ and $W$, respectively. 
	
	\begin{equation*}
		V = \begin{bmatrix} V_1  \\ V_2 \end{bmatrix} = \begin{bmatrix}	0  \\ 1 \end{bmatrix}, \ \ W = \begin{bmatrix} 1 \\ \frac{\delta \xi - m(1+\alpha \xi) - \epsilon (1 + \alpha \xi)^2}{\delta (\delta \xi - m (1 + \alpha \xi))(1+\alpha \xi - \xi)} \end{bmatrix}.
	\end{equation*}
	
	Let us denote system (\ref{4iscd}) as $H = \left[ \begin{matrix} F  \\ G \end{matrix} \right].$
	Thus, 
	
	\begin{eqnarray*}
		H_{\xi} (E_2; \xi^*) &=&  \begin{bmatrix} 0 \\ \frac{\delta (\delta \xi - m (1+\alpha\xi))}{\epsilon (1 + \alpha \xi)^3} \end{bmatrix}, \\
		DH_\xi(E_2; \xi^*)V &=& 
		\begin{bmatrix}
			\frac{\partial F_\xi}{\partial x} & \frac{\partial F_\xi}{\partial y} \\
			\frac{\partial G_\xi}{\partial x} & \frac{\partial G_\xi}{\partial y}
		\end{bmatrix}
		\begin{bmatrix}
			V_1 \\
			V_2
		\end{bmatrix}
		_{(E_2; \xi^*)} =
		\begin{bmatrix}	0 \\
			\frac{\delta}{(1 + \alpha \xi)^2}
		\end{bmatrix}, \\
		D^2 H(E_2; \xi^*)(V, V) &=&
		\begin{bmatrix}
			\frac{\partial^2 F}{\partial x^2} V_1^2 + 2 \frac{\partial^2 F}{\partial x \partial y} V_1 V_2 + \frac{\partial^2 F}{\partial y^2} V_2^2 \\
			\frac{\partial^2 G}{\partial x^2} V_1^2 + 2 \frac{\partial^2 G}{\partial x \partial y} V_1 V_2 + \frac{\partial^2 G}{\partial y^2} V_2^2
		\end{bmatrix}_{(E_2, \xi^*)} = \begin{bmatrix} 0 \\ - 2 \epsilon	\end{bmatrix}.
	\end{eqnarray*}
	
	If $\delta \neq m \alpha$, we have
	
	\begin{eqnarray*}
		W^{T} H_{\xi} (E_2; \xi^*) &=&  \frac{\delta \xi - m(1+\alpha \xi) - \epsilon (1 + \alpha \xi)^2}{\epsilon (1 +\alpha \xi)^3 (1+\alpha \xi - \xi)}  \neq 0, \\
		W^{T} [DH_\xi(E_2; \xi^*)V] &=& \frac{\delta \xi - m(1+\alpha \xi) - \epsilon (1 + \alpha \xi)^2}{(\delta \xi - m (1 + \alpha \xi))(1+\alpha \xi - \xi) (1+\alpha \xi)^2} \neq 0, \\
		W^{T} [D^2 H(E_2; \xi^*)(V, V) ] &=& - 2 \epsilon \left(\frac{\delta \xi - m(1+\alpha \xi) - \epsilon (1 + \alpha \xi)^2}{\delta (\delta \xi - m (1 + \alpha \xi))(1+\alpha \xi - \xi)}\right) \neq 0.
	\end{eqnarray*}
	By the Sotomayor's theorem \cite{perko2013differential}, the system (\ref{4iscd}) undergoes a saddle-node bifurcation around $E_2$ at $\xi = \xi^*$.
\end{proof}

In \autoref{saddle4iscd}, the saddle-node bifurcation around the another axial equilibrium $E_2$ with respect to $\xi$ is discussed. For this same set of parameter values, both the equilibria $E_0$ and $E_2$ exist and move towards each other as $\xi$ reduces. Further at $\xi = 3.0$, both ($E_0$ and $E_2$) collide and become $E_0$ which ensures the happening of saddle-node bifurcation. When $\xi < 3.0$, then there does not exist any prey-free equilibrium $E_2$. We depict the equilibria and their stability for the following set of parameter values, $\gamma = 1.0,\ \alpha = 1.0,\ \epsilon = 0.5,\ \delta = 8.0,\ m=6.0, \ \omega = 4.0$. 

\begin{figure}[ht]
	\centering
	\includegraphics[width=\textwidth]{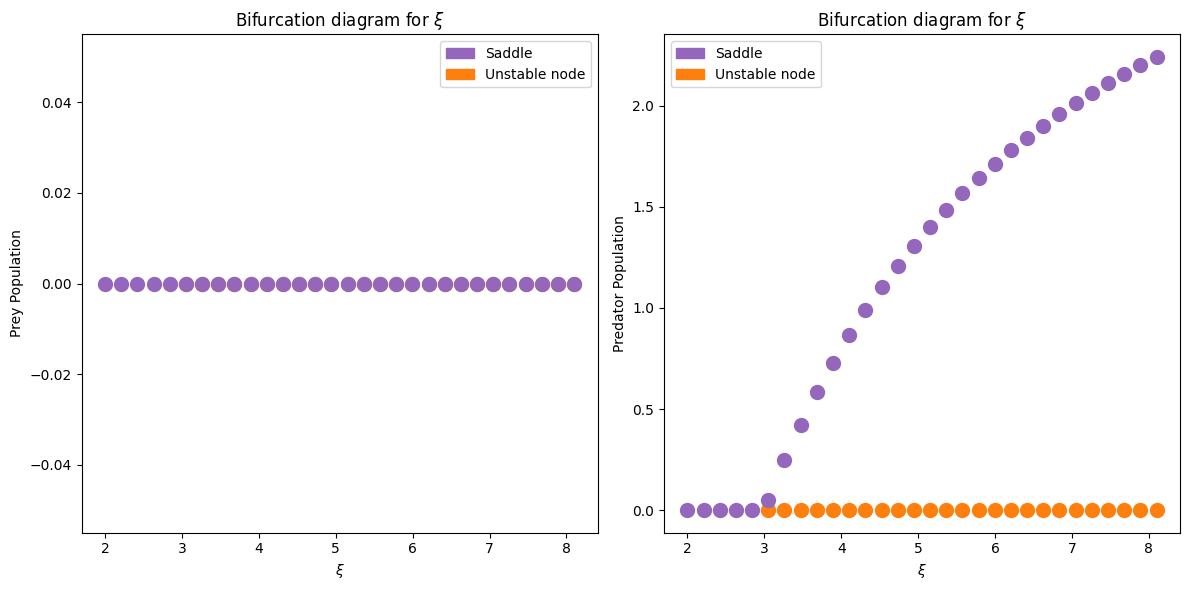}
	\caption{Saddle-node bifurcation diagram around axial equilibrium $E_2 = \left(0,\frac{\delta \xi - m (1 + \alpha \xi)}{\epsilon (1+\alpha \xi)}\right)$ with respect to the quantity of additional food $\xi$.}
	\label{saddle4iscd}
\end{figure}

\subsection{Hopf Bifurcation}

Within this subsection, we derive the conditions for the existence of Hopf bifurcation near the equilibrium point $E^* = \left( x^*, y^* \right)$ using the parameter $\xi$ as the bifurcation parameter.

\begin{thm}
	Let the following conditions be satisfied by the parameters of system (\ref{4iscd}), 
	\begin{itemize}
		\item $\epsilon = \epsilon^* = \frac{1}{y} \left(1 - \frac{x}{\gamma}\right) \left(\frac{x + 2 \omega x^2 (1 + \alpha \xi)}{ x + (\omega x^2 + 1)(1+\alpha \xi)}\right) - \frac{x}{\gamma y} \bigg|_{*},$
		\item $0 < \epsilon < \frac{\delta - m + 2 \omega x^* (\delta \xi - m (1+\alpha \xi))}{2 y^* (2 \omega x^*  (1+\alpha \xi)+1)},$
		\item $x^* \neq \frac{1}{\sqrt{\omega}},$
	\end{itemize}
	then the system (\ref{4iscd}) experiences Hopf bifurcation with respect to intra-specific competition $\epsilon$ about the interior equilibrium point $E^* = (x^*,y^*)$. Here $\epsilon^*$ is the critical intra-specific competition at which the Hopf bifurcation occurs.
\end{thm}

\begin{proof}
	The characteristic equation of Jacobian matrix $J_{E^*}$ is given by 
	
	\begin{equation}
		\lambda^2 - \text{Tr}(J_{E^*}) \lambda + \text{Det}(J_{E^*}) = 0,
	\end{equation}
	where the expression of $\text{Tr}(J_{E^*})$ and $\text{Det}(J_{E^*})$ are given by 
	
	From (\ref{4iscdinttrace}), it is obvious that $\text{Tr}(J_{E^*}) = 0$ when 
	\begin{equation} 
		\begin{split}
			\epsilon = \epsilon^* = \frac{1}{y} \left(1 - \frac{x}{\gamma}\right) \left(\frac{x + 2 \omega x^2 (1 + \alpha \xi)}{ x + (\omega x^2 + 1)(1+\alpha \xi)}\right) - \frac{x}{\gamma y} \bigg|_{*}. 
		\end{split}
	\end{equation}
	
	Since both eigenvalues at Hopf bifurcation point are purely imaginary numbers, we have Det$(J_{E^*}) > 0$. From \autoref{4iscdintdet}, the determinent is positive when $$0 < \epsilon < \frac{\delta - m + 2 \omega x^* (\delta \xi - m (1+\alpha \xi))}{2 y^* (2 \omega x^*  (1+\alpha \xi)+1)}.$$
	
	We also have $\frac{d\left(\text{Tr}(J_{E^*}) \right)}{d \xi} \bigg|_{\xi^*} \neq 0 \implies x^* \neq \frac{1}{\sqrt{\omega}}.$
	
	Therefore, when these conditions are satisfied, the Implicit theorem guarentees the occurrence of Hopf bifurcation at $E^* (x^*,y^*)$ i.e., small amplitude periodic solutions bifurcate from $E^* (x^*,y^*)$ through a Hopf bifurcation.
	
\end{proof}

The existence of Hopf bifurcation with respect to $\epsilon$ is depicted in \autoref{hopf4iscd}. Stable limit cycle is observed around the interior equilibrium when $\epsilon = 0.024$ and the limit cycle disappears when $\epsilon$ is increased to $0.03$. The remaining parameter values are as follows: $\gamma = 15.0,\ \alpha = 0.1,\ \xi = 0.45,\ \delta = 0.45,\ m=0.28, \ \omega = 0.01$.

\begin{figure}[ht]
	\centering
	\includegraphics[width=\textwidth]{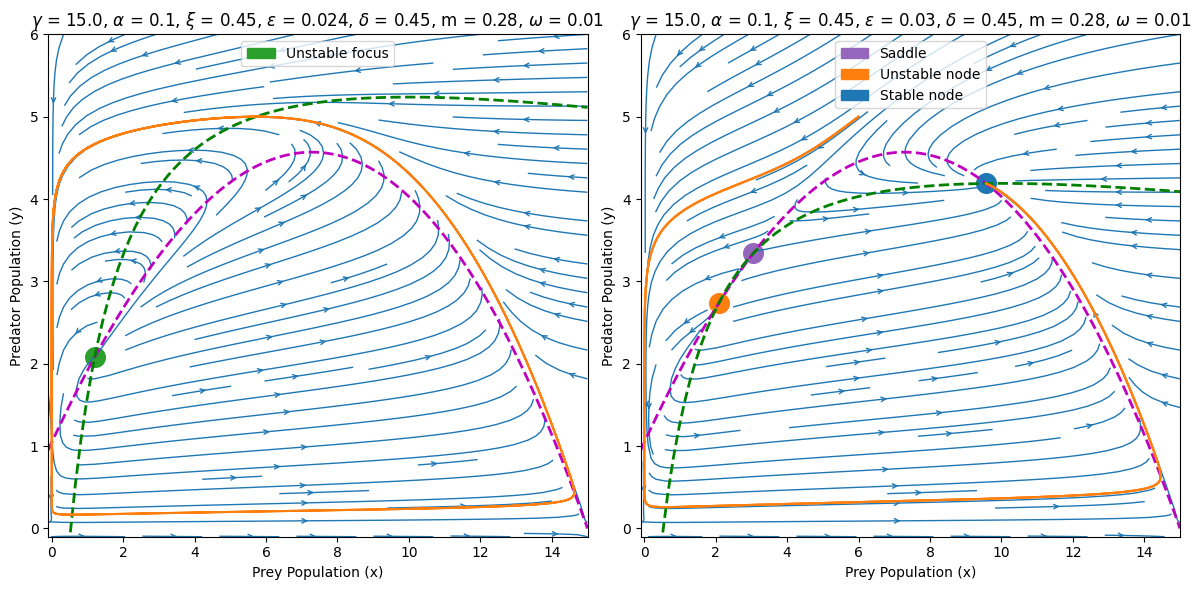}
	\caption{Supercritical Hopf bifurcation diagram with respect to the intra-specific competition $\epsilon$.}
	\label{hopf4iscd}
\end{figure}

\subsection{Cusp bifurcation}

A cusp bifurcation is a codimension-2 bifurcation that occurs when two control parameters interact to create a qualitative change in the system's stability behavior. This bifurcation is observed in \autoref{cusp14iscd} - \autoref{cusp24iscd} for the system (\ref{4iscd}). Here, the horizontal axis corresponds to the quantity (or quality) of additional food provided, and the vertical axis represents the strength of intra-specific competition. Within this parameter space, distinct regions of bistability and monostability emerge, with the cusp bifurcation marking the critical transition point.

The cusp bifurcation point is a unique threshold where the system shifts from exhibiting bistable behavior to monostability. In the bistable region, the system can stabilize at two distinct equilibrium states depending on initial conditions. As parameters vary and cross the cusp point, one of these stable equilibria loses stability, resulting in a monostable regime with a single stable equilibrium. The remaining parameters are as follows: $\gamma = 0.6545,\ \alpha= 0.1,\ \xi = 1.0,\ \delta = 0.6,\ m=0.2,\ \omega = 0.1$.

\begin{figure}[ht]
	\centering
	\includegraphics[width=0.6\textwidth]{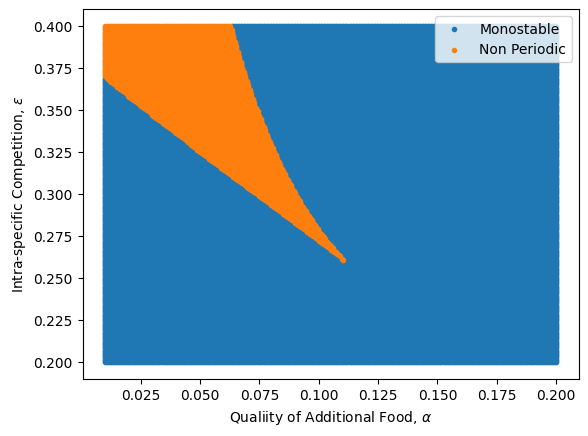}
	\caption{Cusp bifurcation with respect to the qualiity of additional food ($\alpha$) and intra-specific competition ($\epsilon$).}
	\label{cusp14iscd}
\end{figure}

\begin{figure}[ht]
	\centering
	\includegraphics[width=0.6\textwidth]{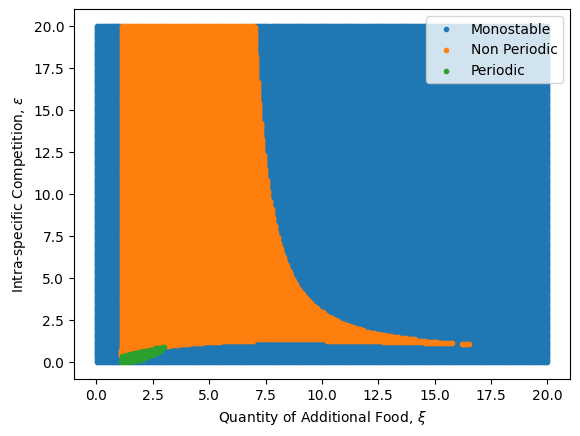}
	\caption{Cusp bifurcation with respect to the quantity of additional food ($\xi$) and intra-specific competition ($\epsilon$).}
	\label{cusp24iscd}
\end{figure}

\section{Global Dynamics}
\label{sec:4iscdglobaldynamics}

In this section, we study the global dynamics of the system (\ref{4iscd}) in the $\alpha$-$\xi$ parameter space. For this study, we divide the parameter space of the system (\ref{4iscd}) in the absence of additional food into three regions. These regions are divided based on the qualitative behaviors of the interior equilibria of the system. They are divided into three regions, namely, $R_1,\ R_2,\ R_3$ corresponding to the space where there is no interior equilibria, unstable interior equilibria and the stable interior equilibria respectively. \autoref{initial4iscd} depicts the phase portrait of the initial system in all three regions. 

\begin{figure}[!ht]
	\centering
	\includegraphics[width=\linewidth]{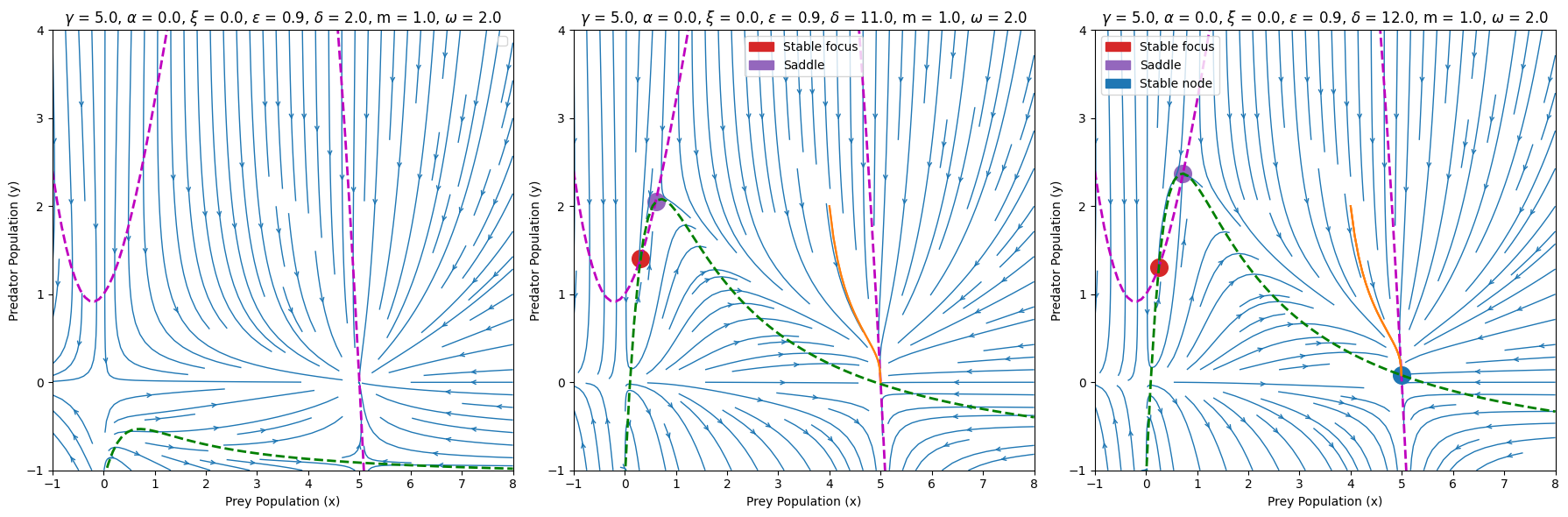}  
	\caption{Dynamics of the system (\ref{4iscd}) in the absence of additional food.}
	\label{initial4iscd}
\end{figure}

Now, in each of these regions, we study the influence of additional food by dividing the $\alpha - \xi$ parameter space into regions based on the following curves. Each of these curves divide the space into two regions based on the qualitative nature of the equilibrium point. 

\begin{equation*}
	\begin{split}
		\text{Bifurcation Curve for } E_0: &\ \phi_1(\alpha,\xi) : \  \delta \xi - m (1 + \alpha \xi) = 0. \\
		\text{Bifurcation Curve for } E_1: &\ \phi_2(\alpha,\xi) : \ \delta \xi - m (1 + \alpha \xi) + \frac{(\delta - m) \gamma}{\omega \gamma^2 + 1} = 0. \\
		\text{Bifurcation Curve for } E_2: &\ \phi_3(\alpha,\xi) : \ \delta \xi - m (1 + \alpha \xi) - \epsilon (1 + \alpha \xi)^2 = 0. \\
		\text{Existence Curve for } E^*: &\ \phi_4(\alpha,\xi) : \ \delta \xi - m (1 + \alpha \xi) + \frac{(\delta - m)}{2 \sqrt{\omega}} = 0. \\			
	\end{split}
\end{equation*} 

\begin{figure}[!ht]
	\centering
	\includegraphics[width=0.7\linewidth]{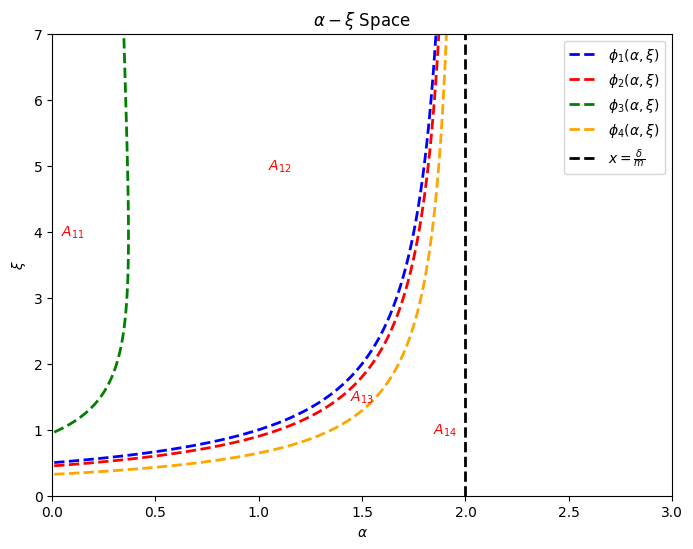} \\
	\caption{Influence of additional food on the system (\ref{4iscd}) when the parameters belong to the region $R_1$.}
	\label{r14iscd}
\end{figure}

\autoref{r14iscd} divides the $\alpha - \xi$ space into $4$ regions. In this region, there is no interior equilibrium in the absence of additional food. As additional food is provided, the prey-free equilibrium $E_2$ is stable in the region $A_{11}$. Also, the predator-free equilibrium $E_1$ is stable in regions $A_{13}$ and $A_{14}$. Provision of additional food leads to the emergence of interior equilibria $E^*$ in the regions $A_{11},\ A_{12},\ A_{13}$. However, provision of additional food in the region $A_{14}$ takes us to the original qualitative behavior as in the case of absence of additional food.

\begin{figure}[!ht]
	\centering
	\includegraphics[width=0.7\linewidth]{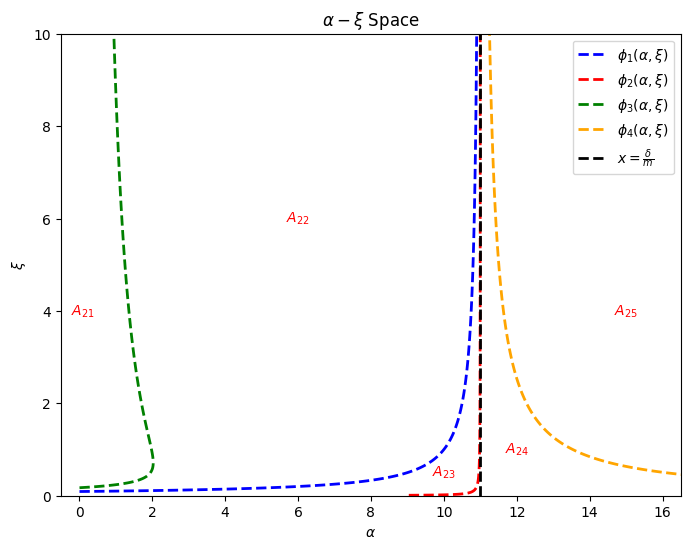} \\
	\caption{Influence of additional food on the system (\ref{4iscd}) when the parameters belong to the region $R_2$.}
	\label{r24iscd}
\end{figure}

\autoref{r24iscd} divides the $\alpha - \xi$ space into $5$ regions. In this region, there is unstable interior equilibrium in the absence of additional food. As additional food is provided, the prey-free equilibrium $E_2$ is stable in the region $A_{21}$. Also, the predator-free equilibrium $E_1$ is stable in regions $A_{24}$ and $A_{25}$. As additional food is provided, the system (\ref{4iscd}) continues to exhibit interior equilibria $E^*$ in the regions $A_{21},\ A_{22},\ A_{23}$ and $A_{24}$. However, provision of additional food in the region $A_{25}$ leads to the disappearance of interior equilibria.

\begin{figure}[!ht]
	\centering
	\includegraphics[width=0.7\linewidth]{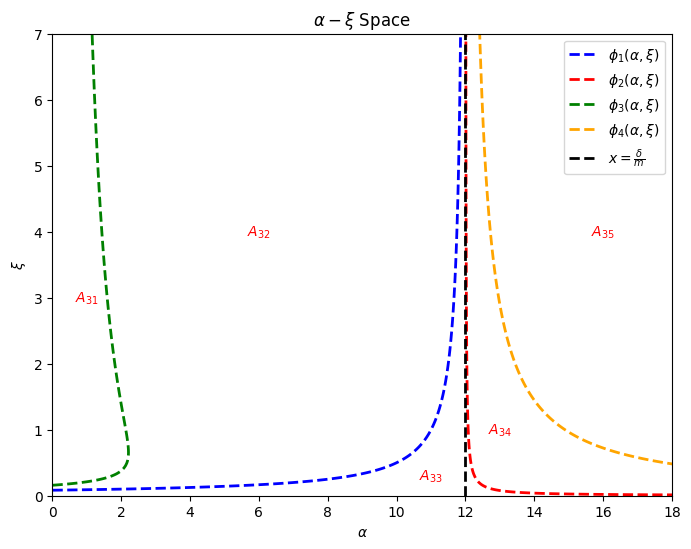} \\
	\caption{Influence of additional food on the system (\ref{4iscd}) when the parameters belong to the region $R_3$.}
	\label{r34iscd}
\end{figure}

\autoref{r34iscd} divides the $\alpha - \xi$ space into $5$ regions. In this region, there is a stable interior equilibrium in the absence of additional food. As additional food is provided, the prey-free equilibrium $E_2$ is stable in the region $A_{31}$. Also, the predator-free equilibrium $E_1$ is stable in regions $A_{34}$ and $A_{35}$. As additional food is provided, the system (\ref{4iscd}) continues to exhibit interior equilibria $E^*$ in the regions $A_{31},\ A_{32},\ A_{33}$ and $A_{34}$. However, provision of additional food in the region $A_{35}$ leads to the disappearance of interior equilibria.

\section{Consequences of providing Additional Food}
\label{sec:4iscdconseq}

\indent In this section, we present the consequences of providing additional food by studying the possibility of existence in three different scenarios: Pest eradication, pest dominance and the coexistence of pest and natural enemies. 

Irrespective of the initial behaviour of the system in the absence of additional food, the provision of additional food in the regions $A_{11},\ A_{21}, \ A_{31}$ leads to a stable pest-free equilibrium $E_2$. Therefore, provision of additional food in these regions can lead to pest eradication. However, the interior equilibria, which represents the coexistence of pest and natural enemies, also exists in these regions. Depending on the nature of interior equilibria, there is a possibility of the bistability in these regions. Therefore, pest eradication in $A_{11},\ A_{21}, \ A_{31}$ depends on the initial condition of pest and natural enemies. 

Now consider the provision of additional food in the regions $A_{14},\ A_{25}$ and $A_{35}$ where there is no interior equilibria. In all these regions, only predator-free equilibrium $E_1$ is stable. Therefore, pest domination is observed in these regions with high quantity of additional food. This could happen possibly due to the abundant availability of additional food which is diverting natural enemies from pests. Hence, arbitrary provision of additional food, in spite of being non-reproductive, does not lead to pest eradication. It can also lead to the completely opposite drastic scenarios for the ecosystem. 

We now discuss the third case where the pest eradication is not possible. However, keeping pest at levels below which they can damage the crop is the next best desirable alternative. Neither of the axial equilibria are stable in the regions $A_{12},\ A_{22},\ A_{23},\ A_{32}$ and $A_{33}$. Therefore, only interior equilibria are stable in these regions leading to the co-existence of pest and natural enemies. 

Lastly, in the regions $A_{13},\ A_{24}$ and $A_{34}$, the interior equilibria exists and the natural enemy-free equilibrium is stable. If interior equilibrium is stable in these regions, then we have bistability in these regions. In these cases, interior equilibrium (coexistence of pest and natural enemies) is desirable to the natural enemy-free equilibrium. Therefore, initial population of pest and natural enemies play a crucial role in the success of bio-control strategies in these regions. 

\autoref{equiplot4iscd} numerically depicts the stability nature of various equilibria that the system exhibits only when additional food and competition terms are altered. This shows the importance of these two terms in the dynamics of the system (\ref{4iscd}). Each frame depicts the existence of $0-3$ interior equilibrium.

\begin{figure}[!ht]
	\includegraphics[width=0.8\textwidth]{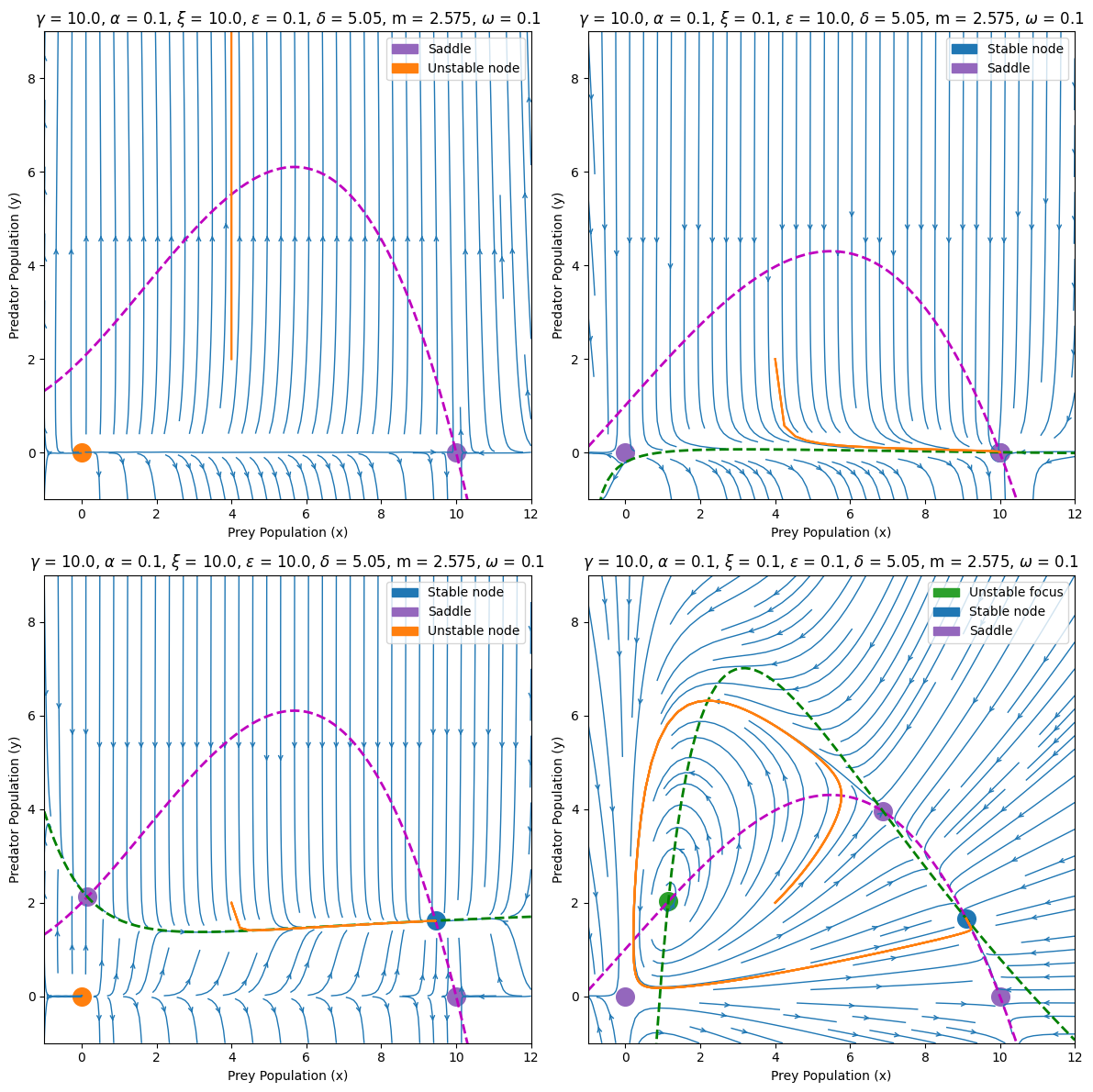}
	\caption{The stability nature of various equilibria of the system (\ref{4iscd}).}
	\label{equiplot4iscd}
\end{figure}

\section{Time-Optimal Control Studies for Holling type-IV Systems} \label{sec:4iscdtimecontrol}

In this section, we formulate and characterise two time-optimal control problems with quality of additional food and quantity of additional food as control parameters respectively. We shall drive the system (\ref{4iscd}) from the initial state $(x_0,y_0)$ to the final state $(\bar{x},\bar{y})$ in minimum time.

\subsection{Quality of Additional Food as Control Parameter}

We assume that the quantity of additional food $(\xi)$ is constant and the quality of additional food varies in $[\alpha_{\text{min}},\alpha_{\text{max}}]$. The time-optimal control problem with additional food provided prey-predator system involving Holling type-IV functional response and intra-specific competition among predators (\ref{4iscd}) with quality of additional food ($\alpha$) as control parameter is given by

\begin{equation}
	\begin{rcases}
		& \displaystyle {\bf{\min_{\alpha_{\min} \leq \alpha(t) \leq \alpha_{\max}} T}} \\
		& \text{subject to:} \\
		& \frac{\mathrm{d} x}{\mathrm{d} t} = x \left(1-\frac{x}{\gamma} \right)- \frac{xy}{(1+\alpha \xi)(\omega x^2 + 1) + x}, \\
		& \frac{\mathrm{d} y}{\mathrm{d} t} = \delta \left( \frac{x + \xi (\omega x^2 + 1)}{(1+\alpha \xi)(\omega x^2 + 1) + x} \right) y - m y - \epsilon y^2, \\
		& (x(0),y(0)) = (x_0,y_0) \ \text{and} \ (x(T),y(T)) = (\bar{x},\bar{y}).
	\end{rcases}
	\label{4iscdalpha0}
\end{equation}

This problem can be solved using a transformation on the independent variable $t$ by introducing an independent variable $s$ such that $\mathrm{d}t = ((1+\alpha \xi)(\omega x^2 + 1) + x) \mathrm{d}s$. This transformation converts the time-optimal control problem ($\ref{4iscdalpha0}$) into the following linear problem.

\begin{equation}
	\begin{rcases}
		& \displaystyle {\bf{\min_{\alpha_{\min} \leq \alpha(t) \leq \alpha_{\max}} S}} \\
		& \text{subject to:} \\
		& \dot{x}(s) = x \left(1 - \frac{x}{\gamma}\right) ((1+\alpha \xi)(\omega x^2 + 1) + x) - x y, \\
		& \dot{y}(s) = \delta (x + \xi (\omega x^2 + 1)) y - ((1+\alpha \xi)(\omega x^2 + 1) + x) (m y + \epsilon y^2), \\
		& (x(0),y(0)) = (x_0,y_0) \ \text{and} \ (x(S),y(S)) = (\bar{x},\bar{y}).
	\end{rcases}
	\label{4iscdalpha}
\end{equation}

Hamiltonian function for this problem (\ref{4iscdalpha}) is given by

\begin{eqnarray*}
	\mathbb{H}(s,x,y,\alpha, p,q) &=& p \left(x \left(1 - \frac{x}{\gamma}\right) ((1+\alpha \xi)(\omega x^2 + 1) + x) - x y\right) \\
	& &+ q \left(\delta (x + \xi (\omega x^2 + 1)) y - ((1+\alpha \xi)(\omega x^2 + 1) + x) (m y + \epsilon y^2)\right) \\
	&=& \left[ p x \left(1 - \frac{x}{\gamma}\right) - q y (m + \epsilon y) \right] (1 + \omega x^2) \xi \alpha \\ 
	& & + p x \left(\left(1 - \frac{x}{\gamma}\right) (x + 1 + \omega x^2) - y\right) \\
	& & + q y \left( \delta (x + (\omega x^2 + 1) \xi) - (x + 1 + \omega x^2) (m + \epsilon y)\right).
\end{eqnarray*}

Here, $p$ and $q$ are costate variables satisfying the adjoint equations 

\begin{equation*}
	\begin{split}
		\dot{p} = & -p \left[x \left(2 - \frac{3x}{\gamma}\right) + (1+\alpha \xi)\left( 1- \frac{2 x}{\gamma} + 3 \omega x^2 - \frac{4\omega x^3}{\gamma}  \right) - y \right] \\
		&  -  q y \left[ \delta - m -\epsilon y + 2\omega x \left(\delta \xi - (m +\epsilon y) (1+\alpha \xi) \right) \right], \\
		\dot{q} = & p x - q \left[ (\delta - m - 2 \epsilon y) x + (\omega x^2 + 1)\left(\delta \xi - (m+2 \epsilon y) (1+\alpha \xi)\right) \right].
	\end{split}
\end{equation*}

Since Hamiltonian is a linear function in $\alpha$, the optimal control can be a combination of bang-bang and singular controls. Since we are minimizing the Hamiltonian, the optimal strategy is given by 

\begin{equation}
	\alpha^*(t) =
	\begin{cases}
		\alpha_{\max}, &\text{ if } \frac{\partial \mathbb{H}}{\partial \alpha} < 0. \\
		\alpha_{\min}, &\text{ if } \frac{\partial \mathbb{H}}{\partial \alpha} > 0.
	\end{cases}
\end{equation}
where
\begin{equation}
	\frac{\partial \mathbb{H}}{\partial \alpha} = \left[ p x \left(1 - \frac{x}{\gamma}\right) - q y (m + \epsilon y) \right] (1 + \omega x^2) \xi .
\end{equation}

This problem (\ref{4iscdalpha}) admits a singular solution if there exists an interval $[s_1,s_2]$ on which $\frac{\partial \mathbb{H}}{\partial \alpha} = 0$. Therefore, 

\begin{equation}
	\frac{\partial \mathbb{H}}{\partial \alpha} = p x \left(1 - \frac{x}{\gamma}\right) - q y (m + \epsilon y) = 0 \textit{ i.e. } \frac{p}{q} = \frac{y (m + \epsilon y)}{x \left(1 -\frac{x}{\gamma}\right)}. \label{4iscdapbyq1}
\end{equation}

Differentiating $\frac{\partial \mathbb{H}}{\partial \alpha}$ with respect to $s$ we obtain 

\begin{equation*}
	\begin{split}
		\frac{\mathrm{d}}{\mathrm{d}s} \frac{\partial \mathbb{H}}{\partial \alpha} = & \frac{\mathrm{d}}{\mathrm{d}s} \left[ \left[ p x \left(1 - \frac{x}{\gamma}\right) - q y (m + \epsilon y) \right] (1 + \omega x^2) \xi \right] \\
		= & \left[ \dot{p} x \left(1 - \frac{x}{\gamma}\right) - \dot{q} y (m + \epsilon y) - q (m + 2 \epsilon y) \dot{y} \right] (1 + \omega x^2) \xi \\
		& + \left[ p \left(1-\frac{2 x}{\gamma} + 3 \omega x^2 - \frac{4 \omega x^3}{\gamma}\right) - 2 \omega q x y (m+\epsilon y) \right] \dot{x}.
	\end{split}
\end{equation*}

Substituting the values of $\dot{x}, \dot{y}, \dot{p}, \dot{q}$ in the above equation and simplifying, we obtain
\begin{equation*}
	\begin{split}
		& \frac{\mathrm{d}}{\mathrm{d}s} \frac{\partial \mathbb{H}}{\partial \alpha} = p x \left[ \frac{xy}{\gamma} - x \left(1 - \frac{x}{\gamma}\right)^2 (1+2\omega x (1+\alpha \xi)) - (m+\epsilon y) y \right] \\
		& - q y \left[ x \left(1 - \frac{x}{\gamma}\right) \left(\delta - m - \epsilon y + 2 \omega x (\delta \xi - (m+\epsilon y) (1+\alpha \xi))\right) + \delta \epsilon y (x + \xi (\omega x^2 + 1))\right].
	\end{split}
\end{equation*}

Along the singular arc, $\frac{\mathrm{d}}{\mathrm{d}s} \frac{\partial \mathbb{H}}{\partial \alpha} = 0$. This implies that 

\begin{equation*}
	\begin{split}
		& p x \left[ \frac{xy}{\gamma} - x \left(1 - \frac{x}{\gamma}\right)^2 (1+2\omega x (1+\alpha \xi)) - (m+\epsilon y) y \right] \\
		= & q y \left[ x \left(1 - \frac{x}{\gamma}\right) \left(\delta - m - \epsilon y + 2 \omega x (\delta \xi - (m+\epsilon y) (1+\alpha \xi))\right) + \delta \epsilon y (x + \xi (\omega x^2 + 1))\right].
	\end{split}
\end{equation*}

and that 
\begin{equation} \label{4iscdapbyq2}
	\frac{p}{q} = \frac{y}{x} \ \frac{x \left(1 - \frac{x}{\gamma}\right) \left(\delta - m - \epsilon y + 2 \omega x (\delta \xi - (m+\epsilon y) (1+\alpha \xi))\right) + \delta \epsilon y (x + \xi (\omega x^2 + 1))}{\frac{xy}{\gamma} - x \left(1 - \frac{x}{\gamma}\right)^2 (1+2\omega x (1+\alpha \xi)) - (m+\epsilon y) y}.
\end{equation}

The solutions of the system of equations (\ref{4iscdapbyq1}) and (\ref{4iscdapbyq2}) gives the switching points of the bang-bang control.

\subsection{Quantity of Additional Food as Control Parameter}

We assume that the quality of additional food $(\alpha)$ is constant and the quantity of additional food varies in $[\xi_{\text{min}},\xi_{\text{max}}]$. The time-optimal control problem with additional food provided prey-predator system involving Holling type-IV functional response and intra-specific competition among predators (\ref{4iscd}) with quantity of additional food ($\xi$) as control parameter is given by

\begin{equation}
	\begin{rcases}
		& \displaystyle {\bf{\min_{\xi_{\min} \leq \xi(t) \leq \xi_{\max}} T}} \\
		& \text{subject to:} \\
		& \frac{\mathrm{d} x}{\mathrm{d} t} = x \left(1-\frac{x}{\gamma} \right)- \frac{xy}{(1+\alpha \xi)(\omega x^2 + 1) + x}, \\
		& \frac{\mathrm{d} y}{\mathrm{d} t} = \delta \left( \frac{x + \xi (\omega x^2 + 1)}{(1+\alpha \xi)(\omega x^2 + 1) + x} \right) y - m y - \epsilon y^2, \\
		& (x(0),y(0)) = (x_0,y_0) \ \text{and} \ (x(T),y(T)) = (\bar{x},\bar{y}).
	\end{rcases}
	\label{4iscdxi0}
\end{equation}

This problem can be solved using a transformation on the independent variable $t$ by introducing an independent variable $s$ such that $\mathrm{d}t = ((1 + \alpha \xi)(\omega x^2 + 1) + x) \mathrm{d}s$. This transformation converts the time-optimal control problem (\ref{4iscdxi0}) into the following linear problem.

\begin{equation}
	\begin{rcases}
		& \displaystyle {\bf{\min_{\xi_{\min} \leq \xi(t) \leq \xi_{\max}} S}} \\
		& \text{subject to:} \\
		& \dot{x}(s) = x \left(1 - \frac{x}{\gamma}\right) ((1 + \alpha \xi)(\omega x^2 + 1) + x) - x y, \\
		& \dot{y}(s) = \delta (x + \xi (\omega x^2 + 1)) y - ((1 + \alpha \xi)(\omega x^2 + 1) + x) (m y + \epsilon y^2), \\
		& (x(0),y(0)) = (x_0,y_0) \ \text{and} \ (x(S),y(S)) = (\bar{x},\bar{y}).
	\end{rcases}
	\label{4iscdxi}
\end{equation}

Hamiltonian function for this problem (\ref{4iscdxi}) is given by
\begin{eqnarray*}
	\mathbb{H}(s,x,y,\xi, p,q) &=& p \left(x \left(1 - \frac{x}{\gamma}\right) ((1 + \alpha \xi)(\omega x^2 + 1) + x) - x y\right) \\
	& &+ q \left(\delta (x + \xi (\omega x^2 + 1)) y - ((1 + \alpha \xi)(\omega x^2 + 1) + x) (m y + \epsilon y^2)\right) \\
	&=& \left[\alpha p x \left(1 - \frac{x}{\gamma}\right) + q \delta y - \alpha q y (m + \epsilon y) \right] (\omega x^2+1) \xi \\ 
	& & + p x \left(\left(1 - \frac{x}{\gamma}\right) (x + 1 + \omega x^2) - y\right) \\
	& & + q y \left( \delta x - (x + 1 + \omega x^2) (m + \epsilon y)\right).
\end{eqnarray*}

Here, $p$ and $q$ are costate variables satisfying the adjoint equations 

\begin{equation*}
	\begin{split}
		\dot{p} = & -p \left[x \left(2 - \frac{3x}{\gamma}\right) + (1+\alpha \xi)\left( 1- \frac{2 x}{\gamma} + 3 \omega x^2 - \frac{4\omega x^3}{\gamma}  \right) - y \right] \\
		&  -  q y \left[ \delta - m -\epsilon y + 2\omega x \left(\delta \xi - (m +\epsilon y) (1+\alpha \xi) \right) \right], \\
		\dot{q} = & p x - q \left[ (\delta - m - 2 \epsilon y) x + (\omega x^2 + 1)\left(\delta \xi - (m+2 \epsilon y) (1+\alpha \xi)\right) \right].
	\end{split}
\end{equation*}

Since Hamiltonian is a linear function in $\xi$, the optimal control can be a combination of bang-bang and singular controls. Since we are minimizing the Hamiltonian, the optimal strategy is given by 

\begin{equation}
	\xi^*(t) =
	\begin{cases}
		\xi_{\max}, &\text{ if } \frac{\partial \mathbb{H}}{\partial \xi} < 0. \\
		\xi_{\min}, &\text{ if } \frac{\partial \mathbb{H}}{\partial \xi} > 0.
	\end{cases}
\end{equation}
where
\begin{equation}
	\frac{\partial \mathbb{H}}{\partial \xi} = \left[\alpha p x \left(1 - \frac{x}{\gamma}\right) + q \delta y - \alpha q y (m + \epsilon y) \right] (\omega x^2+1).
\end{equation}

This problem (\ref{4iscdxi}) admits a singular solution if there exists an interval $[s_1,s_2]$ on which $\frac{\partial \mathbb{H}}{\partial \xi} = 0$. Therefore, 

\begin{equation}
	\frac{\partial \mathbb{H}}{\partial \xi} = \alpha p x \left(1 - \frac{x}{\gamma}\right) + \delta q y - \alpha q y (m + \epsilon y) = 0 \textit{ i.e. } \frac{p}{q} = \frac{\alpha y (m+\epsilon y)-\delta y}{\alpha x \left(1 - \frac{x}{\gamma}\right)}. \label{4iscdxpbyq1}
\end{equation}

Differentiating $\frac{\partial \mathbb{H}}{\partial \xi}$ with respect to $s$ we obtain 

\begin{equation*}
	\begin{split}
		\frac{\mathrm{d}}{\mathrm{d}s} \frac{\partial \mathbb{H}}{\partial \xi} =& \frac{\mathrm{d}}{\mathrm{d}s} \left[ \alpha p x \left(1 - \frac{x}{\gamma}\right) + \delta q y - \alpha q y (m + \epsilon y) \right] \\
		=&\alpha p \left(1 - \frac{2x}{\gamma} \right)\dot{x} + \alpha x \left(1 - \frac{x}{\gamma} \right) \dot{p} + \left(\delta - \alpha (m+\epsilon y) \right) y \dot{q} + \left( \delta - \alpha  (m + 2 \epsilon y) \right) q \dot{y}.
	\end{split}
\end{equation*}

Substituting the values of $\dot{x}, \dot{y}, \dot{p}, \dot{q}$ in the above equation and simplifying, we obtain

\begin{equation*}
	\begin{split}
		\frac{\mathrm{d}}{\mathrm{d}s} \frac{\partial \mathbb{H}}{\partial \xi} = & p x \left[ \frac{\alpha xy}{\gamma} - \alpha x (1+2 \omega x (1+\alpha \xi)) \left(1 - \frac{x}{\gamma}\right)^2 + y (\delta - \alpha (m+\epsilon y))\right] \\
		& + \delta \epsilon q y^2 \left((1-\alpha) x + 1 + \omega x^2 \right).
	\end{split}
\end{equation*}

Along the singular arc, $\frac{\mathrm{d}}{\mathrm{d}s} \frac{\partial \mathbb{H}}{\partial \xi} = 0$. This implies that 

\begin{equation*}
	\begin{split}
		p x \left[ \frac{\alpha xy}{\gamma} - \alpha x (1+2 \omega x (1+\alpha \xi)) \left(1 - \frac{x}{\gamma}\right)^2 + y (\delta - \alpha (m+\epsilon y))\right] & \\
		+ \delta \epsilon q y^2 \left((1-\alpha) x + 1 + \omega x^2 \right) & = 0.
	\end{split}
\end{equation*}

and that 
\begin{equation} \label{4iscdxpbyq2}
	\frac{p}{q} = \frac{y^2}{x} \  \ \frac{\delta \epsilon \left((1-\alpha) x + 1 + \omega x^2 \right)}{\frac{\alpha xy}{\gamma} - \alpha x (1+2 \omega x (1+\alpha \xi)) \left(1 - \frac{x}{\gamma}\right)^2 + y (\delta - \alpha (m+\epsilon y))}.
\end{equation}

The solutions of the system of equations (\ref{4iscdxpbyq1}) and (\ref{4iscdxpbyq2}) gives the switching points of the bang-bang control.

\subsection{Applications to Pest Management}

In this subsection, we simulated the time-optimal control problems (\ref{4iscdalpha}) and (\ref{4iscdxi}) using CasADi in python \cite{CasADi}. We implemented the direct transcription method with multiple shooting in order to solve the time-optimal control problems. In this method, we discretize the control problem into smaller intervals using finite difference integration, specifically the fourth-order Runge-Kutta (RK4) method. By breaking the trajectory into multiple shooting intervals, the state and control variables at each node are treated as optimization variables. The dynamics of the system are enforced as constraints between nodes, allowing for greater flexibility and improved convergence when solving the nonlinear programming problem with CasADi’s solvers.

\begin{figure}[!ht]
	\centering
	\includegraphics[width=\textwidth]{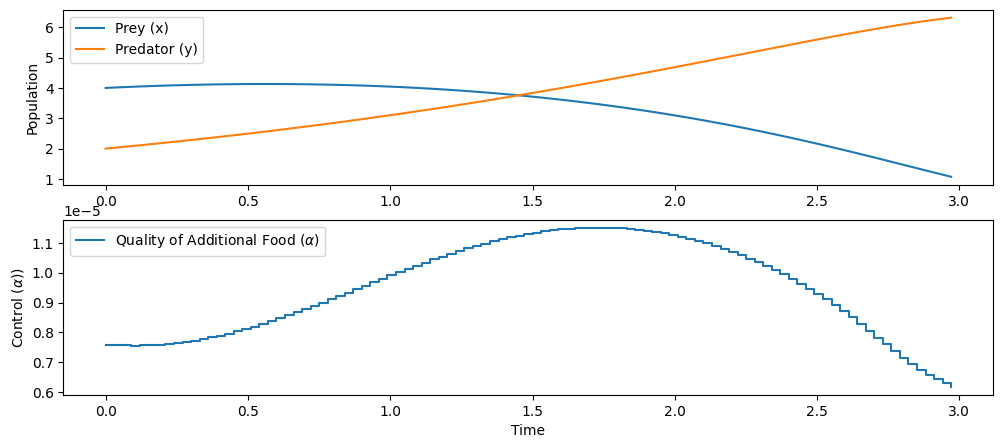}
	\caption{The optimal state trajectories and the optimal control trajectories for the time optimal control problem (\ref{4iscdxi}).}
	\label{c4iscdalpha}
\end{figure}

\begin{figure}[!ht]
	\centering
	\includegraphics[width=\textwidth]{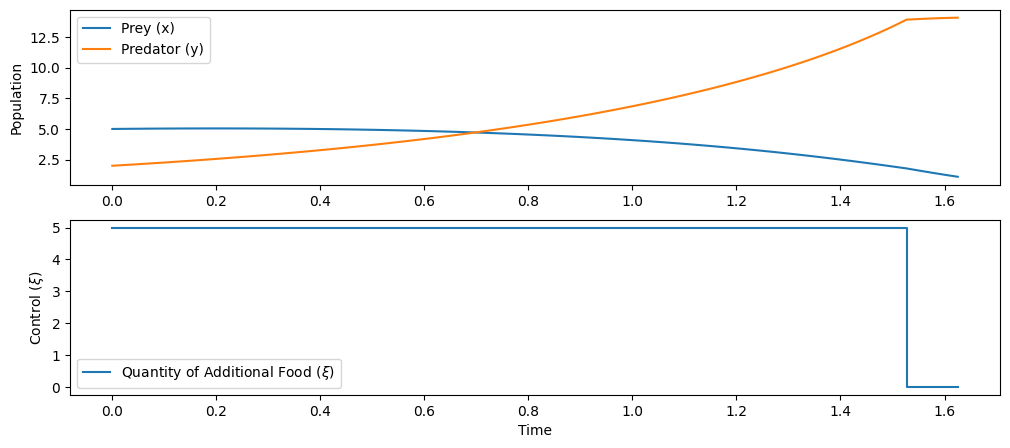}
	\caption{The optimal state trajectories and the optimal control trajectories for the time optimal control problem (\ref{4iscdxi}).}
	\label{c4iscdxi}
\end{figure}

\autoref{c4iscdalpha} illustrates the optimal state and control trajectories for the time-optimal control problem (\ref{4iscdalpha}). The simulation uses the parameter values $\gamma = 8.0,\  \xi = 0.1,\ \delta = 0.96, \ m=0.3,\ \epsilon = 0.01,\ \omega = 0.01$, starting from the initial point $(4,2)$ and reaching the final point $(1,3)$ in an optimal time of $2.97$ units. 

\autoref{c4iscdxi} illustrates the optimal state and control trajectories for the time-optimal control problem (\ref{4iscdxi}). The simulation uses the parameter values $\gamma = 8.0,\  \alpha = 0.1,\ \delta = 0.96, \ m=0.3,\ \epsilon = 0.01,\ \omega = 0.01$, starting from the initial point $(5,2)$ and reaching the final point $(1,4)$ in an optimal time of $1.62$ units. 

\section{Discussions and Conclusions} \label{sec:disc}

\indent This paper studies deterministic prey-predator systems exhibiting Holling type-IV functional responses among the additional food provided predators that exhibit intra-specific competition. To begin with, we proved the positivity and boundedness of solutions. As the model could exhibit atmost $5$ real roots, we proved the condition for the existence of interior equilibria in \autoref{4iscdintcond}. The conditions for stability of various equilibria is presented in \autoref{sec:4iscdstab}. From the qualitative theory of dynamical systems, we observed that the system (\ref{4iscd}) exhibits the trans-critical, saddle, Hopf and two cusp bifurcations. In addition to this, we presented a detailed study on the global dynamics and consequences of providing additional food. Further, we formulated the time-optimal control problems with the objective to minimize the final time in which the system reaches the pre-defined state. Here, we considered the quality and the quantity of additional food as control variables. Using the Pontraygin maximum principle, we characterized the optimal control values. We also numerically simulated the theoretical findings and applied them in the context of pest management.

Some of the salient features of this work include the following. This work captures the commonly observed intra-specific competition among predators for the additional food provided prey-predator system exhibiting Holling type-IV functional responses. A rigorous analysis on the dynamics of this system is presented in this work. In addition to this, this paper also dealt with the novel study of the time-optimal control problems by transforming the independent variable in the control system. This work has been an initial attempt dealing with the time optimal control studies for prey-predator systems involving intra-specific competition among predators. This initial exploratory research will further lead to a more sophisticated model and rigorous analysis in the context of sensitivity and estimation of parameters, controllability and observability in the future works.

\subsection*{Financial Support: }
This research was supported by National Board of Higher Mathematics(NBHM), Government of India(GoI) under project grant - {\bf{Time Optimal Control and Bifurcation Analysis of Coupled Nonlinear Dynamical Systems with Applications to Pest Management, \\ Sanction number: (02011/11/2021NBHM(R.P)/R$\&$D II/10074).}}

\subsection*{Conflict of Interests Statement: }
The authors have no conflicts of interest to disclose.

\subsection*{Ethics Statement:} 
This research did not required ethical approval.

\subsection*{Acknowledgments}
The authors dedicate this paper to the founder chancellor of SSSIHL, Bhagawan Sri Sathya Sai Baba. The contributing author also dedicates this paper to his loving elder brother D. A. C. Prakash who still lives in his heart.

\printbibliography

@article{V4Acta,
	title = {Additional {{Food Supplements}} as a {{Tool}} for {{Biological Conservation}} of {{Biosystems}} in the {{Presence}} of {{Inhibitory Effect}} of the {{Prey}}},
	author = {Vamsi, D. K. K. and Kanumoori, Deva Siva Sai Murari and Chhetri, Bishal},
	date = {2020-09},
	journaltitle = {Acta Biotheoretica},
	shortjournal = {Acta Biotheor},
	volume = {68},
	number = {3},
	pages = {321--355},
	issn = {0001-5342, 1572-8358},
	doi = {10.1007/s10441-019-09371-x}
}

@article{V4DEDS,
	title = {Biological {{Conservation}} of {{Living Systems}} by {{Providing Additional Food Supplements}} in the {{Presence}} of {{Inhibitory Effect}}: {{A Theoretical Study Using Predator}}--{{Prey Models}}},
	shorttitle = {Biological {{Conservation}} of {{Living Systems}} by {{Providing Additional Food Supplements}} in the {{Presence}} of {{Inhibitory Effect}}},
	author = {Srinivasu, P. D. N. and Vamsi, D. K. K. and Aditya, I.},
	date = {2018-01},
	journaltitle = {Differential Equations and Dynamical Systems},
	shortjournal = {Differ Equ Dyn Syst},
	volume = {26},
	number = {1-3},
	pages = {213--246},
	issn = {0971-3514, 0974-6870},
	doi = {10.1007/s12591-016-0344-4}
}

@Article{CasADi,
	author = {Joel A E Andersson and Joris Gillis and Greg Horn
	and James B Rawlings and Moritz Diehl},
	title = {{CasADi} -- {A} software framework for nonlinear optimization
	and optimal control},
	journal = {Mathematical Programming Computation},
	volume = {11},
	number = {1},
	pages = {1--36},
	year = {2019},
	publisher = {Springer},
	doi = {10.1007/s12532-018-0139-4}
}

@article{howard1998gronwall,
	title={The gronwall inequality},
	author={Howard, Ralph},
	journal={lecture notes},
	year={1998}
}

@book{perko2013differential,
	title={Differential equations and dynamical systems},
	author={Perko, Lawrence},
	volume={7},
	year={2013},
	publisher={Springer Science and Business Media}
}

@book{kot2001elements,
  title={Elements of mathematical ecology},
  author={Kot, Mark},
  year={2001},
  publisher={Cambridge University Press}
}

@book{metz2014dynamics,
  title={The dynamics of physiologically structured populations},
  author={Metz, Johan A and Diekmann, Odo},
  volume={68},
  year={2014},
  publisher={Springer}
}

@article{v2,
  title={Additional food supplements as a tool for biological conservation of predator-prey systems involving type III functional response: A qualitative and quantitative investigation},
  author={Srinivasu, P D N and Vamsi, D K K and Ananth, V S},
  journal={Journal of theoretical biology},
  volume={455},
  pages={303--318},
  year={2018},
  publisher={Elsevier}
}

@article{ananthcmb,
  title={An Optimal Control Study with Quantity of Additional food as Control in Prey-Predator Systems involving Inhibitory Effect},
  author={Ananth, V S and Vamsi, D K K},
  journal={Computational and Mathematical Biophysics},
  volume={9},
  number={1},
  pages={114--145},
  year={2021},
  publisher={De Gruyter Open Access}
}

@article{ananth2021influence,
  title={Influence of quantity of additional food in achieving biological conservation and pest management in minimum-time for prey-predator systems involving Holling type III response},
  author={Ananth, V S and Vamsi, D K K},
  journal={Heliyon},
  volume={7},
  number={8},
  pages={e07699},
  year={2021},
  publisher={Elsevier}
}

@article{ananth2022achieving,
  title={Achieving Minimum-Time Biological Conservation and Pest Management for Additional Food provided Predator--Prey Systems involving Inhibitory Effect: A Qualitative Investigation},
  author={Ananth, V S and Vamsi, D K K},
  journal={Acta Biotheoretica},
  volume={70},
  number={1},
  pages={1--51},
  year={2022},
  publisher={Springer}
}

@article{prakash2023stochastic,
  title={Stochastic optimal and time-optimal control studies for additional food provided prey--predator systems involving Holling type III functional response},
  author={Prakash, Daliparthi Bhanu and Vamsi, Dasu Krishna Kiran},
  journal={Computational and Mathematical Biophysics},
  volume={11},
  number={1},
  pages={20220144},
  year={2023},
  publisher={De Gruyter}
}

@article{prakash9stochastic,
  title={Stochastic Time-Optimal Control and Sensitivity Studies for Additional Food provided Prey-Predator Systems involving Holling Type-IV Functional Response},
  author={Prakash, D Bhanu and Vamsi, D K K},
  journal={Frontiers in Applied Mathematics and Statistics},
  volume={9},
  pages={1122107},
  publisher={Frontiers}
}

\end{document}